\documentclass[11pt]{amsart}

\usepackage{amsmath,amssymb,amscd,amsthm,amsfonts}

\begin{document}
\theoremstyle{plain}
\newtheorem{thm}{Theorem}[section]
\newtheorem*{thm*}{Theorem}
\newtheorem{prop}[thm]{Proposition}
\newtheorem*{prop*}{Proposition}
\newtheorem{lemma}[thm]{Lemma}
\newtheorem{cor}[thm]{Corollary}
\newtheorem*{conj*}{Conjecture}
\newtheorem*{cor*}{Corollary}
\newtheorem{defn}[thm]{Definition}
\newcommand\sL{{\mathcal{L}}}\newcommand\aMtilde{{\tilde \aM}}
\newcommand\Brl{{\mathcal{B}}}
\newcommand\Rl{{\mathbb{R}}}
\newcommand\Bd{B}
\newcommand\cC{{\ucal{C}}}
\newcommand\aM{{\mathcal{M}}}\newcommand\Cx{{\mathbb{C}}}
\newcommand\aN{{\mathcal{N}}}
\newtheorem{cond}{Condition}
\theoremstyle{definition}
\newtheorem*{defn*}{Definition}
\newtheorem{rems}[thm]{Remarks}
\newtheorem*{rems*}{Remarks}
\newtheorem*{proof*}{Proof}
\newtheorem*{not*}{Notation}
\newcommand{\npartial}{\slash\!\!\!\partial}
\newcommand{\Heis}{\operatorname{Heis}}
\newcommand{\Solv}{\operatorname{Solv}}
\newcommand{\Spin}{\operatorname{Spin}}
\newcommand{\SO}{\operatorname{SO}}
\newcommand{\ind}{\operatorname{ind}}
\newcommand{\Index}{\operatorname{index}}
\newcommand\Mat{\operatorname{Mat}}
\newcommand\Mn{\operatorname{Mat}_n}
\newcommand{\ch}{\operatorname{ch}}
\newcommand{\rank}{\operatorname{rank}}
\newcommand{\dom}{\operatorname{dom}}
\newcommand{\abs}[1]{\lvert#1\rvert}
 \newcommand{\A}{{\mathcal A}}
        \newcommand{\D}{{\mathcal D}}\newcommand{\HH}{{\mathcal H}}
        \newcommand{\LL}{{\mathcal L}}
        \newcommand{\B}{{\mathcal B}}
        \newcommand{\F}{{\mathcal F}}
        \newcommand{\K}{{\mathcal K}}
\newcommand{\oo}{{\mathcal O}}
         \newcommand{\PP}{{\mathcal P}}
        \newcommand{\s}{\sigma}
\newcommand{\al}{\alpha}
        \newcommand{\coker}{\operatorname{coker}}
        \newcommand{\p}{\partial}
        \newcommand{\dd}{|\D|}
        \newcommand{\n}{\parallel}
\newcommand{\bma}{\left(\begin{array}{cc}}
\newcommand{\ema}{\end{array}\right)}
\newcommand{\bca}{\left(\begin{array}{c}}
\newcommand{\eca}{\end{array}\right)}
\newcommand\clsp{\overline{\operatorname{span}}}
\newcommand\T{\mathbb T}
\newcommand\Aut{\operatorname{Aut}}
\newcommand\End{\operatorname{End}}
\newcommand{\Res}{\operatorname{Res}}
\newcommand{\sr}{\stackrel}
\newcommand{\da}{\downarrow}
\newcommand{\tD}{\tilde{\D}}
\newcommand\Tr{\operatorname{Tr}}
\newcommand\Ind{\operatorname{Index}}

        \newcommand{\R}{\mathbf R}
        \newcommand{\C}{\mathbb C}
        \newcommand{\h}{\mathbf H}
\newcommand{\Z}{\mathbb Z}
\newcommand{\N}{\mathbb N}
\newcommand{\tto}{\longrightarrow}
\newcommand{\ben}{\begin{displaymath}}
        \newcommand{\een}{\end{displaymath}}
\newcommand{\be}{\begin{equation}}
\newcommand{\ee}{\end{equation}}

        \newcommand{\bean}{\begin{eqnarray*}}
        \newcommand{\eean}{\end{eqnarray*}}
\newcommand{\nno}{\nonumber\\}
\newcommand{\bea}{\begin{eqnarray}}
        \newcommand{\eea}{\end{eqnarray}}

\newcommand\cross[1]{\rlap{\hskip#1pt\hbox{$-$}}}
        \newcommand\intcross{\cross{0.3}\int}
        \newcommand\bigintcross{\cross{2.3}\int}

\newcommand\RR{{\mathcal R}}

\newcommand{\supp}[1]{\operatorname{#1}}
\newcommand{\norm}[1]{\parallel\, #1\, \parallel}
\newcommand{\ip}[2]{\langle #1,#2\rangle}
\setlength{\parskip}{.3cm}
\newcommand{\nc}{\newcommand}
\nc{\nt}{\newtheorem} \nc{\gf}[2]{\genfrac{}{}{0pt}{}{#1}{#2}}
\nc{\mb}[1]{{\mbox{$ #1 $}}} \nc{\real}{{\mathbb R}}
\nc{\comp}{{\mathbb C}} \nc{\ints}{{\mathbb Z}}
\nc{\Ltoo}{\mb{L^2({\mathbf H})}} \nc{\rtoo}{\mb{{\mathbf R}^2}}
\nc{\slr}{{\mathbf {SL}}(2,\real)} \nc{\slz}{{\mathbf
{SL}}(2,\ints)} \nc{\su}{{\mathbf {SU}}(1,1)} \nc{\so}{{\mathbf
{SO}}} \nc{\hyp}{{\mathbb H}} \nc{\disc}{{\mathbf D}}
\nc{\torus}{{\mathbb T}}
\newcommand{\tk}{\widetilde{K}}
\newcommand{\boe}{{\bf e}}\newcommand{\bt}{{\bf t}}
\newcommand{\vth}{\vartheta}
\newcommand{\CGh}{\widetilde{\CG}}
\newcommand{\db}{\overline{\partial}}
\newcommand{\tE}{\widetilde{E}}
\newcommand{\tr}{\mbox{tr}}
\newcommand{\ta}{\widetilde{\alpha}}
\newcommand{\tb}{\widetilde{\beta}}
\newcommand{\txi}{\widetilde{\xi}}
\newcommand{\hV}{\hat{V}}
\newcommand{\IC}{\mathbf{C}}
\newcommand{\IZ}{\mathbf{Z}}
\newcommand{\IP}{\mathbf{P}}
\newcommand{\IR}{\mathbf{R}}
\newcommand{\IH}{\mathbf{H}}
\newcommand{\IG}{\mathbf{G}}
\newcommand{\CC}{{\mathcal C}}
\newcommand{\CS}{{\mathcal S}}
\newcommand{\CG}{{\mathcal G}}
\newcommand{\CL}{{\mathcal L}}
\newcommand{\CO}{{\mathcal O}}
\newcommand\Ev{\operatorname{Ev}}
\newcommand\Ad{\operatorname{Ad}}
\nc{\ca}{{\mathcal A}} \nc{\cag}{{{\mathcal A}^\Gamma}}
\nc{\cg}{{\mathcal G}} \nc{\chh}{{\mathcal H}} \nc{\ck}{{\mathcal
B}} \nc{\cl}{{\mathcal L}} \nc{\cm}{{\mathcal M}}
\nc{\cn}{{\mathcal N}} \nc{\cs}{{\mathcal S}} \nc{\cz}{{\mathcal
Z}} \nc{\cM}{{\mathcal M}}
\nc{\sind}{\sigma{\rm -ind}}
\newcommand{\la}{\langle}
\newcommand{\ra}{\rangle}

\renewcommand{\labelitemi}{{}}

\title{Twisted cyclic theory, equivariant $KK$ theory and KMS States.}
\parskip=0.0cm
\author[A. L. Carey]{Alan L. Carey}
\address{Mathematical Sciences Institute, Australian National University,
Canberra, ACT 0200, Australia} \email{acarey@maths.anu.edu.au}
\author[S. Neshveyev]{Sergey Neshveyev}
\address{Department of Mathematics, University of Oslo,
P.O. Box 1053 Blindern, 0316 Oslo, Norway}
\email{sergeyn@math.uio.no}
\author[R. Nest]{Ryszard Nest}
\address{Department of Mathematics, Copenhagen University, Universitetsparken
5, 2100 Copenhagen, Denmark} \email{rnest@math.ku.dk}
\author[A. Rennie]{Adam Rennie}
\address{Mathematical Sciences Institute, Australian National University,
Canberra, ACT 0200, Australia} \email{rennie@maths.anu.edu.au}
\parskip=0.3cm

\maketitle

\vspace{-8pt}
\centerline{{\em Dedicated to the memory of Gerard Murphy}}

\vspace{4pt}
\centerline{Abstract} Recently, examples of an index theory for KMS
states of circle actions were discovered, \cite{CPR2,CRT}. We show that these
examples are not isolated. Rather there is a general framework
in which we use KMS states for circle actions on a $C^*$-algebra
$A$ to construct
Kasparov modules and  semifinite spectral triples.
By using a residue construction analogous to that
used in the  semifinite local index formula
we associate to these triples
a twisted cyclic cocycle on a dense subalgebra of $A$.
This cocycle pairs with the equivariant $KK$-theory of the mapping cone
algebra for the inclusion of the fixed point algebra of the circle action in $A$.
The pairing is expressed in terms of spectral flow between a pair of unbounded
self adjoint operators that are Fredholm in the semifinite sense.
A novel aspect of our work is the discovery of an eta cocycle that forms a part of our
twisted residue cocycle.
To illustrate our theorems we observe firstly that they incorporate
the results in \cite{CPR2,CRT} as special cases. Next we use
the Araki-Woods III$_\lambda$ representations of the
Fermion algebra to show that there are examples which are not Cuntz-Krieger systems.
\parskip=0.3cm

\section{Introduction}
\subsection{Background}
This paper presents an extension of noncommutative geometry and index
theory to the purely
infinite or type III case, that is, where we are looking at
situations in which there are no faithful traces
and we wish to replace them by KMS states. It
exploits some ideas from semifinite
noncommutative geometry, a recent extension of the standard type I
theory of Connes \cite{C}
which was begun in \cite{CP1}. Our main reference for the analytic
part of these
ideas is \cite{CPS2}, however there is more background in
\cite{BeF, CPRS1,CPRS2,CPRS3}.

We build on two interesting examples. In  \cite{CPR2,CRT} it was
discovered that there exist refined invariants of the Cuntz algebra
and the algebra $SU_q(2)$ which arise from KMS states.  In the
former case there is a canonical circle action (the gauge action)
and a  unique associated KMS state whose GNS representation is type
III while in the latter case the Haar state satisfies the KMS condition
with respect to a circle action. It was shown that certain unitaries in matrix
algebras over the Cuntz algebras and $SU_q(2)$
can be used to form a new type of
$K$-group. This group then pairs with twisted cyclic cocycles   to
produce real valued invariants. In \cite{CPR2} this construction was
termed  `modular index theory'. Not all unitaries in these
algebras define elements of the new $K$-group; only those satisfying a
side condition formulated in terms of the modular group of the KMS
state. These index pairings have not been seen before because they
use in an essential way the semifinite index theory from
\cite{CPS2,CPRS2,CPRS3}.

The nature of this modular theory in these examples is
mysterious. The
objective of this paper is to put the examples into a
general framework so that, rather than being isolated phenomena,
the modular index pairings in \cite{CPR2,CRT} can be seen to arise from
a more fundamental principle.
The germ of the idea comes from  \cite{CPR1} where it was observed
that the examples of semifinite noncommutative geometries discovered
in \cite{PRen} lead to classes in the $KK$-theory of a mapping cone
algebra. What we find in this paper is a more general framework that
uses equivariant $KK$-theory of mapping cone algebras. The
constructions we employ arise very naturally for a broad class of
$C^*$-algebras admitting states (or weights) that are KMS for circle
actions. We remark that although we do not discuss examples of
the case where we start with a weight on a $C^*$-algebra here,
we know from work in progress with M. Marcolli that such examples
exist and are of considerable independent interest.

\subsection{Summary of results}
Our basic data consists of a $C^*$-algebra $A$, together with a
strongly continuous action of the circle $\T$ by $*$-automorphisms $
\s\colon \T\to \Aut(A). $ We let $F$ denote the fixed point subalgebra
$A^\sigma$ of $A$, $\Phi$ the conditional expectation
$$
A\ni a\mapsto \frac{1}{2\pi }\int_\T \s_t (a) dt \in F ,
$$
and set $A_k=\{ a\in A \mid \s_t (a)=e^{ikt}a\}$. For the more
algebraic part of this paper (Section 2) we make a `spectral subspace
assumption' (SSA) that the ideals $F_k =\overline{A_k A_k^*}$ are
complemented in $F$. This generalises the notion of full spectral
subspaces. When we deal with purely analytic formulae in subsequent
Sections we can drop this SSA.
On the analytic side we assume there is a (possibly unbounded)
semifinite, norm lower semicontinuous, faithful, positive functional
$\phi\colon A\to\C$ satisfying the KMS$_\beta$ condition  for the
circle action $\s$ and $\beta\ne0$.

Associated to the pair
$(A,\s )$ there is an  unbounded Kasparov module $
({}_AX_F, \D)
$ constructed as follows. Define the $F$-valued scalar product $(a|b)_R =\Phi (a^* b)$
and
${}_AX_F$ as the corresponding $C^*$-module completion of $A$.
Then $\D$ is the generator
of the action of $\s$ on ${}_AX_F$ induced by the action of $\s$
on $A$.
The pair $({}_AX_F, \D)$ defines a class
$
[\D]:=[({}_AX_F, \D)]\in KK_1^\T (A, F).
$

Let $M=M(F,A)$ denote the mapping cone of the inclusion $F\subset
A$. The mapping cone extension
$$
0\rightarrow C_0(\mathbb{R}, A)\stackrel{\iota}{\rightarrow} M\stackrel{ev}{\rightarrow} F
\rightarrow 0
$$
gives us an exact sequence in equivariant $KK$-theory, part of which is
$$
KK_0^\T(M,F)\stackrel{\iota^*}{\rightarrow}KK_1^\T(A,F)\stackrel{\delta}{\to} KK_1^\T(F,F).
$$
Since $\D$ commutes with the left action of $F$, we have $\delta[\D]=0$, and so by exactness, 
there is a class $[\hat\D]\in KK_0^\T(M,F)$ with $\iota^*[\hat\D]=[\D]$.
In the text we will give an explicit construction of such a  $KK$-class
$[\hat\D]$ using the results of \cite{CPR1}. 

The cycles $[\D]$ and $[\hat\D]$ define, via the Kasparov product, two index maps:
$$
K_1^\T(A)
\xrightarrow{\Ind_\D}K_0^\T (F)
\ \mbox{
and }\
K_0^\T (M)\xrightarrow{\Ind_{\hat{\D}}} K_0^\T (F).
 $$
compatible with $\iota_* \colon K_1^\T (A)\rightarrow K_0^\T (M)$. 

We
shall also introduce a closely related homomorphism into the representation ring of
$\mathbb T$:
$$
sf\colon K_0^\T
(M)\to\Rl[\chi,\chi^{-1}],
$$
which we call the `equivariant spectral flow'. We will explain how the
constructions of a modular index pairing for the Cuntz algebras and
$SU_q(2)$ obtained previously in \cite{CPR2,CRT} can be explained (and
generalized) in terms of $\Ind_{\hat\D}$ and $sf$ together with the map
$\tau_*\colon K_0(F)\to\Rl$ induced by the trace~$\tau=\phi|_F$ and the
evaluation at $\chi=e^{-\beta}$ map $\Rl[\chi,\chi^{-1}]\to\Rl$ (Theorem
\ref{thm:ind-mod-uni}). Furthermore, the modular index pairing can be
computed quite explicitly {\em for a certain subgroup of $K_0^\T(M)$}
using an analytic formula for semifinite spectral flow
(Theorem~\ref{thm:analytic-index}).

As a new example we discuss the case of Araki-Woods factors which are
obtained from KMS states for the gauge action on the Fermion algebra.



\subsection{The theorems}\label{sub:thms}
For the reader's convenience we give details of the results here.
Let $\HH_\phi =L^2 (A,\phi )$, let $\pi_\phi \colon A\rightarrow
\B(\HH_\phi )$ denote the GNS representation and $\cn$ the commutant of
$J_\phi \pi_\phi(F) J_\phi$ in $\B(\HH_\phi )$.  Then $\cn$ is a
semifinite von Neumann algebra and $A\simeq \pi_\phi(A)\subset\cn$ with
a positive, faithful, semifinite trace $\mbox{Tr}_\phi$ (see Lemma
\ref{spatial trace}). Let also $\D$ denote the (self-adjoint) extension
of the operator $\D$ introduced above to $\HH_\phi$. Let $\Phi_k$ be
the projection onto the $k^{th}$ spectral subspace of $\D$ for each
$k\in\mathbb Z$. Finally, set $P=\chi_{[0,\infty)} (\D)$. Denote by
$\A$ the algebra consisting of finite sums of $\s$-homogeneous elements
in the domain $\dom(\phi)$ of $\phi$. We also put $\F=\A\cap
F=\dom(\tau)$, where $\tau=\phi|_F$.

We denote the unitization of our algebras by a superscripted $\tilde{}
$. Every class in $K^\T_0(M)$ has a representative $v$ such that
$v\in(\A^\sim\otimes B(\HH_U))^{\s\otimes\Ad U}$, $vv^*$ and $v^*v$ are
in $\F^\sim\otimes B(\HH_U)$, and $vv^*=v^*v$ modulo $\F\otimes
B(\HH_U)$, where $U\colon \T\to B(\HH_U)$ is a finite dimensional
unitary representation (Lemma~\ref{lem:rep}). Henceforth we restrict to
such $v$. Then we define the equivariant spectral flow as follows.
Denote by~$\chi^n$ the one-dimensional representation $t\mapsto
e^{int}$ and write the representation ring of $\T$ over $\mathbb R$ as $\Rl
[\chi,\chi^{-1}]$.
Let $Q_n\colon\HH\otimes\HH_U\to\HH\otimes\HH_U$ be the projection onto the
$\chi^n$-homogeneous component
and let
$$
sf_n(v)=(\Tr_\phi\otimes\Tr)((v^*v-vv^*)Q_n(P\otimes1))\in\Rl.
$$
We show that $sf_n(v)=0$ for all but a finite number of $n\in\Z$ and
introduce the $\T$-equivariant spectral flow  $sf$ on $K^\T_0(M))$ with
values in $\Rl[\chi,\chi^{-1}]$ defined on the class $[v]$ of $v$ by
$$
sf([v])=\sum_{n\in\Z}sf_n(v)\chi^n.
$$
Let $K^I(M)$ be the subgroup of $K^\T_0(M)$ generated by partial
isometries whose homogeneous components are partial isometries.

\begin{thm}
If the spectral subspace assumption is satisfied, the equivariant
spectral flow coincides with the composition
$$
K^\T_0(M)\xrightarrow{-\Ind_{\hat\D}}K_0^\T(F)=K_0(F)[\chi,\chi^{-1}]
\xrightarrow{\tau_*}\Rl[\chi,\chi^{-1}],
$$
where $\tau_*$ denotes the homomorphism $K_0(F)\to\Rl$
defined by the trace $\tau$. For elements of $K^I(M)$ this was done explicitly in \cite{CPR1}.
\end{thm}




It is natural to ask whether there is an analytic spectral flow formula
that computes the equivariant spectral flow. There is an immediate
obstacle, seen in examples: $(1+\D^2)^{-\frac{1}{2}}$ is almost never
finitely summable with respect to $\Tr_\phi$. To obtain a spectral
triple we use the method of \cite{CPR2}.  We construct on $\mathcal N$
a faithful semifinite normal weight $\phi_\D\equiv \mbox{Tr}_\phi
(e^{-\beta\D/2}\cdot e^{-\beta\D/2})$ such that

$\bullet$ the modular automorphism group $\sigma^{\phi_\D}$ of
$\phi_\D$ is implemented by a one parameter unitary group whose
generator is $\D$ and, moreover,
$\sigma_t^{\phi_\D}\vert_\A=\sigma_{-\beta t}$ for
 all $t\in\Rl$,

 \vspace{-4pt}

$\bullet$ $\phi_\D$ restricts to a faithful normal semifinite trace on
the fixed point algebra $\cM$ of $\sigma^{\phi_\D}$, $\D$ is affiliated
to $\cM$, and $[\D,a]$ extends to a bounded operator (in $\cn$) for all
$a$ in a dense subalgebra of $A$,

\vspace{-4pt}

$\bullet$  for all $f\in F\cap\mbox{dom}(\phi)$ and all $\lambda$ in
the resolvent set of $\D$, the operator $f(\lambda-\D)^{-1}$ belongs to
the ideal $\K(\cM,\phi_\D)$  of compact operators in $\cM$ relative to
$\phi_\D$.

The problem now is that since $A$ is not contained in $\cM$, we do not
have an immediate definition of a spectral flow for partial isometries
in $A$. This has led to the definition of a new group $K_1(A,\s)$
\cite{CPR2}, which is closely related to $K^I(M)$. Namely, a partial
isometry $v\in A$ (or in matrices over $A$) is called {\bf modular} if
$[\D,v]$ is bounded and $vQ v^*\in\cM=\cn^\s$ for every spectral
projection of~$\D$. There is a semigroup $K_1(A,\sigma)$ defined as the
homotopy classes of modular partial isometries in $\Mat_\infty
(A^\sim)=\cup_n \Mat_n(A^\sim)$. Via the Grothendieck construction we
will henceforth use the same notation for the corresponding group. The
main reason to define this group is that for modular partial isometries
$v$, $(Pvv^* ,vPv^* )$ is a Fredholm pair in the semifinite sense
\cite{BCPRSW} in $(\cM ,\phi_\D)$ and hence there is a well-defined
{\em analytic spectral flow} along the path $t\mapsto (1-t) (2Pvv^*-1) +t
(2vPv^*-1)$, $t\in [0,1]$. This is equal ({see \cite{BCPRSW}, Section
6) to the semifinite spectral flow $sf_{\phi_\D}(vv^*\D, v\D v^*)$
along the linear path joining $vv^* \D$ to~$v\D v^*$.

Let $v\in A$ be a modular partial isometry.  Then it can be shown that
the decomposition $v=\sum v_k , v_k \in A_k$ is finite and every $v_k$
is a partial isometry. We can consider $v_k$ as an operator
$\HH_\phi[k]\to\HH_\phi$, where $\HH_\phi[k]$ coincides with $\HH_\phi$
as a space, but the representation of $\T$ is tensored by $\chi^k$.
Then $v_k$ defines a class $\ll\!\! v_k\!\!\gg$ in $K_0^\T(M)$. There
is thus a well defined homomorphism $T\colon K_1 (A,\s ) \rightarrow
K^I(M) $ given by $ v\mapsto \sum_k \ll\!\! v_k\!\!\gg. $

In Section 4 we obtain the following relationship between the
equivariant spectral flow $sf$ with respect to $\Tr_\phi$ and
$sf_{\phi_\D}(vv^*\D, v\D v^*)$.
\begin{thm}
The spectral flow $sf_{\phi_\D}(vv^*\D, v\D v^*)$
for modular partial isometries is the composition of the maps
$$
K_1(A,\sigma)\xrightarrow{T} K^\T_0(M)
\xrightarrow{sf}\Rl[\chi,\chi^{-1}]\xrightarrow{\Ev(e^{-\beta})}\Rl,
$$
where $\Ev(e^{-\beta})$ is the evaluation at $\chi=e^{-\beta}$.
\end{thm}

In Section 5 we turn to the question of providing a direct analytic formula for
$sf_{\phi_\D}(vv^*\D, v\D v^*)$.

\begin{thm}
Let $v\in\A$ be a modular partial isometry. Then
$sf_{\phi_{\D}}(vv^*\D,v\D v^*)$ is given by
$$\Res_{r=1/2}\left(r\mapsto \phi_\D(v[\D,v^*](1+\D^2)^{-r})+
\frac{1}{2}\int_1^\infty \phi_{\D}((\s_{-i\beta}(v^*)v-vv^*)\D
(1+s\D^2)^{-r})s^{-1/2}ds\right).
$$
\end{thm}

In previous papers \cite{CPR2,CRT} it was clear from the numerical
values computed for the spectral flow for particular choices of modular
partial isometries $v$ that the mapping cone $K$-theory was playing a
role. The preceding two theorems explain exactly these previously
somewhat mysterious numerical values.

The above spectral flow formula turns out to be related to twisted
cyclic cohomology as follows. Ignoring eta correction terms one can argue by analogy with
the standard case of tracial weights (see \cite{CPS2})
to define a $\T$-equivariant Chern character computing the equivariant
spectral flow using a Dixmier functional $\Tr_{\phi,\omega}$ by
$$
\phi_1(g;a_0,a_1)=\frac{1}{2}\Tr_{\phi,\omega}(ga_0[\D,a_1]),
$$
where $g$ is an element of a group algebra of $\T$. In Section 5 we can make the definition of 
$\Tr_{\phi,\omega}$ precise using a residue formula and also define a twisted cocycle $\phi_1$
where we choose $g=e^{-\beta}$ considered as an
element of the complexification of $\T$. Namely, we have:
\begin{thm}
(i)  If $(A,\phi,\s)$ has full spectral subspaces then for all
$a_0,a_1,\in\A$ the residue
$$\phi_1(a_0,a_1):=\Res_{r=1/2}\phi_\D(a_0[\D,a_1](1+\D^2)^{-r})$$
exists and equals $\phi(a_0[\D,a_1])$. It defines a twisted cyclic
cocycle on $\A$ with twisting $\s_{-i\beta}$, and for any modular
partial isometry $v$
$$sf_{\phi_\D}(vv^*,v\D v^*)=\phi_1(v,v^*).$$

(ii) For general circle actions, the bilinear functional on $\A$ given
by
$$
\psi^r (a_0,a_1 )=
\phi_\D(a_0[\D,a_1](1+\D^2)^{-r})+\frac{1}{2}
\int_1^\infty
\phi_{\D}((\s_{-i\beta}(a_1)a_0-a_0a_1)\D(1+s\D^2)^{-r})s^{-1/2}ds
$$
depends holomorphically on $r$ for $\Re(r)>1/2$ and modulo functions
which are holomorphic for $\Re(r)>0$ is a function valued
$\s_{-i\beta}$-twisted $(b,B)$-cocycle.
\end{thm}

We have encountered a similar situation in~\cite{NT}. There the
$SU_q(2)$-equivariant Chern character of the Dirac operator on the
quantum sphere was evaluated at a special element $\rho\in
U_q({\mathfrak su}_2)$ and produced an explicitly computable twisted
cyclic cocycle. In our current situation the evaluation at
$g=e^{-\beta}$ is needed to improve summability whereas there the purpose was
somewhat opposite: the evaluation at~$\rho$ prevented the dimension
drop and produced a $2^+$-summable spectral triple in the twisted sense
instead of a $0^+$-summable one. In both cases however we see that
twisted cohomology appears as an analytically manageable part of
an equivariant cohomology.

We make the observation that our constructions and resulting index theorems are 
not related to the approach of Connes and Moscovici in \cite{CM}.

\section{The construction of equivariant $KK$ classes from a circle action}
\label{sec:KK-kms}

\subsection {A Kasparov module from a circle action} \label{ss:KK}
Let $A$ be a $C^*$-algebra, $\sigma\colon\T\to\Aut(A)$ a
strongly continuous action of the circle. It will
be convenient to consider $\sigma$ as a $2\pi$-periodic
one-parameter group of automorphisms.
We denote by $F$ the fixed point algebra $\{a\in A\mid
\sigma_t(a)=a\ \ \forall t\in \Rl\}$. Since $\T$ is a compact group,
the map
$$\Phi\colon A\to F,\qquad \Phi(a)=\frac{1}{2\pi}\int^{2\pi}_0\sigma_t(a)dt$$
is a faithful conditional expectation. Next define an
$F$-valued inner product on $A$ by
$(a|b)_R:=\Phi(a^*b).$
The properties of $\Phi$ allow us to see that this is a
(pre)-$C^*$-inner product on $A$, and so we may complete~$A$ in the
topology determined by the norm
$\Vert a\Vert_X^2=\Vert(a|a)_R\Vert_F$
to obtain a $C^*$-module, for the right action of $F$.

\begin{defn}\label{Fmod}
We let $X=\overline{A}$ be the $C^*$-module completion of $A$ with
inner product $(\cdot|\cdot)_R$.
\end{defn}

The circle action is defined on the dense subspace $A\subset X$ and
extends to a unitary action on $X$. The $F$-module $X$
is a full $F$-module for the right 
inner product.
For $k\in\Z$, denote the eigenspaces of the action $\s$ by
$$ A_k=\{a\in A: \sigma_t(a)=e^{ikt}a\ \ \mbox{for all}\ t\in \Rl\}.$$
Then $F=A_0$, which guarantees the fullness of $X$ 
over $F$. Also
$A$ is a $\Z$-graded algebra in an obvious way,
$A_{-k}=(A_k)^*$ and in particular, each~$A_k$ is an $F$-module.
Note that the norm on $A_k$ defined by the above inner product
coincides with the $C^*$-norm.  We denote by $X_k$ 
the space $A_k$ considered as a closed
submodule of~$X$.
For  $k\in \Z$ we set $F_k=\overline{A_kA_k^*}$. Some of our results
using $KK$-theoretic constructions require the following assumption.

\begin{defn} The action $\sigma$ on
$A$ satisfies the {\bf Spectral Subspace Assumption} (SSA) if $F_k$ is a
complemented ideal in $F$ for every $k\in\Z$. Equivalently, the
representation $\pi_k\colon F\to \End_F(X_k)$ given by left
multiplication satisfies $\pi_k(F)=\pi_k(F_k)$ (then $\ker\pi_k$ is the
complementary ideal to $F_k$).
\end{defn}

There is a special case of this assumption which is well known, namely
$A$ is said to have {\bf full spectral subspaces} if $F_k=F$ for all
$k\in\Z$. The gauge action on the Cuntz algebras $\mathcal O_n$
provides examples where fullness holds. The quantum group $SU_q(2)$
with its Haar state and associated circle action is an example of an
algebra satisfying the SSA but not having full spectral
subspaces~\cite{CRT}.
\begin{lemma}\label{lem:left-full} If  $\overline{A_1A_1^*}=
\overline{A_1^*A_1}=A_0$, the modules $X_k$ and $\overline{X}_k$
are full for all $k\in\Z$.
\end{lemma}
\begin{proof}
Observe that as $A_1A_1^*\subset A_0$, we have
$A_1=\overline{A_0A_1}$. So if $A_0=\overline{A_1^*A_1}$, by
induction we get $A_0=\overline{(A_1^k)^*A_1^k}$ for $k\ge1$. Since
$A^k_1\subset A_k$, we conclude that $X_k$ is full. Similarly, if
$k\le-1$ then $(A_1^*)^{-k}\subset A_k$, so
$A_0=\overline{A_1A_1^*}$ implies that $X_k$ is full.
\end{proof}
Note that as $A_0A_k$ is dense in $A_k$, we always have
$\overline{FA}=A$ and similarly $\overline{AF}=A$ (this also follows
from the existence of a $\sigma$-invariant approximate unit in $A$). As
there are many examples of circle actions which are not full but
satisfy the SSA we will develop the theory in this generality in
the present Section.

Next we remark that
the general theory of $C^*$-modules (or Hilbert modules) is
discussed in many places and we will use \cite{L,RW}. For a right
$C^*$-$B$-module $Y$, we let $\End_B(Y)$ be the $C^*$-algebra (for
the operator norm) of adjointable endomorphisms, $\End_B^0(Y)$ the
ideal of compact endomorphisms, which is the completion of
the finite rank endomorphisms:
$\End^{00}_B(Y)$. The latter is
generated by the rank one endomorphisms $\Theta_{x,y}$, $x,y\in Y$,
defined by $\Theta_{x,y}z=x(y|z)_R$, $z\in Y $.

For each $k\in\Z$, the projection onto the $k$-th spectral subspace
for the circle action is defined by an operator $\Phi_k$ on $X$  via
$$\Phi_k(x)=
\frac{1}{2\pi} \int^{2\pi}_0e^{-ikt}\sigma_t(x)dt,\ \ x\in X.$$ The
range of $\Phi_k$ is the submodule $X_k$. These ranges give us the
natural $\Z$-grading of $X$. The operators $\Phi_k$ are adjointable
endomorphisms of the $F$-module $X$ such that
$\Phi_k^*=\Phi_k=\Phi_k^2$ and $\Phi_k\Phi_l=\delta_{k,l}\Phi_k$. If
$K\subset\Z$ then the sum $\sum_{k\in K}\Phi_k$ converges strictly to a
projection in the endomorphism algebra, \cite{PRen}. In particular sum
$\sum_{k\in\Z}\Phi_k$ converges strictly to the identity operator on
$X$.

The following Lemma is the key step in obtaining a Kasparov module.
\begin{lemma}
For a circle action on $A$ the following conditions are equivalent:\\
(i) the action satisfies the SSA;\\
(ii) for all $a\in A$
and $k\in\Z$, the endomorphism $a\Phi_k$ of the right $F$-module
$X$ is compact.
\end{lemma}
\begin{proof}
Assume the action satisfies the SSA. If $x,y\in A_k$ and $z\in X$, then
$$\Theta_{x,y}z=x\Phi(y^*z)=x\Phi(y^*z_k)=xy^*z_k=xy^*\Phi_kz.$$
Thus $\Theta_{x,y}=xy^*\Phi_k$. It follows that $a\Phi_k$ is compact
for any $a\in A_kA_k^*$. Since $A_kA_k^*$ is dense in $F_k$, we see
that $f\Phi_k$ is compact for any $f\in F_k$, and hence $f\Phi_k$ is
compact for any $f\in F$ by the SSA. But then $af\Phi_k$ is compact for
any $f\in F$ and $a\in A$. Since $AF$ is dense in $A$, we can
approximate $b\Phi_k$ for any $b\in A$ by (compact) endomorphisms of
the form $af\Phi_k$.

Conversely, assume $f\Phi_k$ is compact for some $f\in F$, so that
$f\Phi_k$ can be approximated by finite sums of operators
$\Theta_{x,y}$, $x,y\in X_k$. We have seen, however, that
$\Theta_{x,y}=xy^*\Phi_k$, and so $\pi_k(f)$ is in $\pi_k(F_k)$.
Therefore if $f\Phi_k$ is compact for all $f\in F$ and $k\in\Z$, the
SSA is satisfied.
\end{proof}



Since we have the circle action defined on $X$, we may use the
generator of this action to define an unbounded operator $\D$. We
will not define or study $\D$ from the generator point of view,
instead taking a more bare-hands approach. It is easy to check that
$\D$ as defined below is the generator of the circle action.
The theory of unbounded operators on $C^*$-modules that we require
is all contained in Lance's book, \cite[Chapters 9,10]{L}. We quote
the following definitions (adapted to our situation).

\begin{defn} \cite{L} Let $Y$ be a right $C^*$-$B$-module. A densely defined
unbounded operator $$\D\colon {\rm dom}\ \D\subset Y\to Y$$ is a
$B$-linear operator defined on a dense $B$-submodule ${\rm dom}\
\D\subset Y$. The operator $\D$ is closed if the graph $ G(\D)=\{(x, \D
x):x\in{\rm dom}\ \D\}$ is a closed submodule of $Y\oplus Y$.
\end{defn}

If $\D\colon \mbox{dom}\ \D\subset Y\to Y$ is densely defined and
unbounded, define a submodule
$$\mbox{dom}\ \D^*:=\{y\in Y:\exists z\in Y\ \mbox{such that}\
 \forall x\in\mbox{dom}\ \D, (\D x|y)_R=(x|z)_R\}.$$
Then for $y\in \mbox{dom}\ \D^*$ define $\D^*y=z$. Given
$y\in\mbox{dom}\ \D^*$, the element $z$ is unique, so $\D^*\colon
\mbox{dom}\D^*\to Y$, $\D^*y=z$ is well-defined, and moreover is
closed.

\begin{defn} \cite{L}
Let $Y$ be a right $C^*$-$B$-module. A densely defined unbounded
operator $\D$ is symmetric if for all $x,y\in{\rm dom}\ \D$ we have $
(\D x|y)_R=(x|\D y)_R$. A symmetric operator $\D$ is self-adjoint if
${\rm dom}\ \D={\rm dom}\ \D^*$ (and so $\D$ is necessarily closed). A
densely defined unbounded operator $\D$ is regular if $\D$ is closed,
$\D^*$ is densely defined, and $1+\D^*\D$ has dense range.
\end{defn}
The extra requirement of regularity is necessary in the $C^*$-module
context for the continuous functional calculus, and is not
automatic, \cite[Chapter 9]{L}. With these definitions in hand, we
return to our $C^*$-module $X$.
The following can be proved just as in \cite[Proposition 4.6]{PRen}, or equivalently by
observing that the operator $\D$ is presented in diagonal form.

\begin{prop}\label{CstarDee}
Let $X$ be the right $C^*$-$F$-module of Definition \ref{Fmod}.
Define $X_\D\subset X$ to be the linear space
$$ X_\D=
\{x=\sum_{k\in\Z}x_k\in X:\Vert\sum_{k\in\Z}k^2(x_k|x_k)_R\Vert<\infty\}.$$
For $x=\sum_{k\in\Z}x_k\in X_\D$ define
$ \D x=\sum_{k\in\Z}kx_k.$
Then $\D\colon X_\D\to X$ is a self-adjoint regular operator on $X$.
\end{prop}

There is a continuous functional calculus for self-adjoint regular
operators, \cite[Theorem 10.9]{L}, and we use this to obtain spectral
projections for $\D$ at the $C^*$-module level. Let $f_k\in C_c({\Rl})$
be $1$ in a small neighbourhood of $k\in{\Z}$ and zero on
$(-\infty,k-1/2]\cup[k+1/2,\infty)$. Then it is clear that
$ \Phi_k=f_k(\D).$
That is the spectral projections of $\D$ are the same as the
projections onto the spectral subspaces of the circle action.

\begin{lemma}\label{compactendo}
If the SSA holds , then for all $a\in A$, the operator
$a(1+\D^2)^{-1/2}$ is a compact endomorphism of the $F$-module~$X$.
\end{lemma}

\begin{proof}
Since $a\Phi_k$ is a compact endomorphism for all $a\in A$, and
$a\Phi_k$, $a\Phi_m$ have orthogonal initial spaces, the sum
$$a(1+\D^2)^{-1/2}=\sum_{k\in\Z}(1+k^2)^{-1/2}a\Phi_k$$
converges in norm to a compact endomorphism.
\end{proof}

\begin{prop}\label{Kasmodule}
If the SSA holds, the pair $(X,\D)$ is an unbounded Kasparov module
defining a class in $KK_1(A,F)$.
\end{prop}

\begin{proof}
We will use the approach of \cite[Section 4]{K}.
Let $V=\D(1+\D^2)^{-1/2}$. We need to show
that various operators belong to $\End^0_F(X)$. First, $V-V^*=0$, so
$a(V-V^*)$ is compact for all $a\in A$. Also
$a(1-V^2)=a(1+\D^2)^{-1}$ which is compact from Lemma
\ref{compactendo} and the boundedness of $(1+\D^2)^{-1/2}$. Finally,
we need to show that $[V,a]$ is compact for all $a\in A$. First we
suppose that $a=a_m$ is homogenous for the circle action. Then \bean
[V,a]&=&[\D,a](1+\D^2)^{-1/2}-\D(1+\D^2)^{-1/2}[(1+\D^2)^{1/2},a]
(1+\D^2)^{-1/2}\nno &=&b_1(1+\D^2)^{-1/2}+Vb_2(1+\D^2)^{-1/2},\eean
where $b_1=[\D,a]=ma$ and $b_2=[(1+\D^2)^{1/2},a]$. Provided that
$b_2(1+\D^2)^{-1/2}$ is a compact endomorphism, Lemma
\ref{compactendo} will show that $[V,a]$ is compact for all
homogenous $a$. So consider the action of
$[(1+\D^2)^{1/2},a](1+\D^2)^{-1/2}$ on $x=\sum_{k\in\Z}x_k$. We find
\begin{align}
\sum_{k\in\Z}[(1+\D^2)^{1/2},a](1+\D^2)^{-1/2}x_k&= \sum_{k\in\Z}
\left((1+(m+k)^2)^{1/2}-(1+k^2)^{1/2}\right)(1+k^2)^{-1/2}a x_k\nonumber\\
&=\sum_{k\in\Z}f_{m}(k)a \Phi_kx.\label{limit}\end{align} The
function \ben f_{m}(k)=\left((1+(m+k)^2)^{1/2}-(1+k^2)^{1/2}\right)
(1+k^2)^{-1/2}\een goes to $0$  as $k\to\pm\infty$, and as the
$a_m\Phi_k$ are compact with orthogonal ranges, the sum in
(\ref{limit}) converges in the operator norm on endomorphisms and so
converges to a compact endomorphism. For  $a\in A$  a finite sum of
homogenous terms, we apply the above reasoning to each term in the
sum to find that $[(1+\D^2)^{1/2},a](1+\D^2)^{-1/2}$ is a compact
endomorphism.

Now let $a\in A$ be the norm limit of a Cauchy sequence
$\{a_i\}_{i\geq 0}$ where each $a_i$ is a finite sum of homogenous
terms. Then
$$\Vert[V,a_i-a_j]\Vert_{\End}\leq 2\Vert a_i-a_j\Vert_{\End}\to 0,$$
so the sequence $[V,a_i]$ is also Cauchy in norm, and so the limit
is compact.
\end{proof}


\begin{cor}
If the SSA holds, the pair $(X,\D)$ defines a class in the equivariant
$KK$-group $KK^\T_1(A,F)$.
\end{cor}

The proof of the corollary is obvious from the constructions.

\subsection{The equivariant constructions for the mapping cone algebra}

{}From the unbounded Kasparov $A$-$F$-module $(X,\D)$, we shall
construct a new equivariant Kasparov $M(F,A)$-$F$-module
$(\hat X,\hat\D)$. By pairing the class of the module
$(\hat{X},\hat{\D})$  with elements of $K^\T_0(M(F,A))$
we then get a map $K^\T_0(M(F,A))\to K^\T_0(F)$. Here $M(F,A)$
is the mapping cone $C^*$-algebra for the inclusion
$F\hookrightarrow A$ defined by
$$ M(F,A)=\{f\colon [0,\infty)\to A\,:\,f\in C_0([0,\infty),A),\,f(0)\in F\}.$$
The mapping cone algebra carries the circle action coming from
the circle action on $A$.

In \cite{Put}, Putnam showed that the $K_0$ group of $M(F,A)$ is given
by homotopy classes of partial isometries $v\in A^\sim\otimes \Mat_k(\C)$
with $vv^*,\,v^*v\in F^\sim\otimes \Mat_k(\C)$. Before summarising the
construction of $(\hat{X},\hat{\D})$ from~\cite{CPR1}, we adapt
Putnam's description of the $K$-theory of the mapping cone to an
equivariant setting.

Denote by $V^\T(F,A)$ the set of
$\sigma\otimes\mbox{Ad} U$ invariant partial isometries $v\in
A^\sim\otimes \B(\HH)$, where $U\colon\T\to \B(\HH)$ is some finite
dimensional unitary representation (which varies with $v$ but which we
denote generically by $U$),
such that
$vv^*$ and $v^*v$ belong to $F^\sim\otimes\B(\HH)$.

Consider the equivalence relation on $V^\T(F,A)$ generated by the following two
conditions: two
invariant partial isometries $v_1,\,v_2$ are equivalent if they are joined by a
homotopy consisting of invariant partial isometries or if they are a
pair of the form $v, v\oplus p$, where $p$ is an invariant
projection in $F^\sim\otimes \B(\HH)$ for some $\T$-module~$\HH$. For
$v\in V^\T(F,A)$ we define a projection $p_v$ in a matrix algebra
over the unitization of the mapping cone algebra $M(F,A)$ by
$$p_v(t)=\bma 1-vv^*+\frac{t^2vv^*}{1+t^2}& -iv\frac{t}{1+t^2}\\
iv^*\frac{t}{1+t^2}& \frac{v^*v}{1+t^2}\ema.$$ Then $[p_v]-\left[\bma 1
& 0\\ 0 & 0\ema\right]$ is an element of $K_0^\T(M(F,A))$ (in
particular $p_v$ is a $\s\otimes \Ad U$ invariant projection, for a
suitable representation $U$ of $\T$ coming from the representation
associated with $v$), and the map
$$
v\mapsto [p_v]-\left[\bma 1 & 0\\ 0 & 0\ema\right]
$$
is a bijection of $V^\T(F,A)/\!\!\sim\,\,$ onto $K_0^\T(M(F,A))$. Note
that Putnam considers the non-equivariant case, but the reader can
easily check that all his proofs and constructions carry over to the
$\T$-equivariant case. A general class in $K_0^\T(M(F,A))$ is denoted
by $[v]$ or equivalently by $[p_v]-[1]$.

The group  $K^\T_0(M(F,A))$ is a module over the representation ring of
$\T$, which we identify with the ring $\RR_\T=\Rl[\chi,\chi^{-1}]$ of
Laurent polynomials with real coefficients; therefore $\chi^n$
denotes the one-dimensional representation $t\mapsto e^{int}$. For a
$\T$-module $\HH$ we denote by $\HH[n]$ the module with the same
underlying space but with the action tensored with $\chi^n$. Now in
terms of partial isometries, the $\RR_\T$-module structure on
$K^\T_0(M(F,A))$ is described as follows: if $v\in
(A^\sim\otimes\B(\HH))^\T$ then $\chi[v]$ is the class of the partial
isometry $v$ considered as an element of $(A^\sim\otimes
\B(\HH[1]))^\T$.

The construction of $(\hat X,\hat\D)$ follows \cite{CPR1}, where
a $C^*$-algebra analogue of the Atiyah-Patodi-Singer (APS)
theory,~\cite{APS1}, was described.
We take as our starting point the equivariant
Kasparov module $(X,\D)$ coming from the circle action $\s$ on the
$C^*$-algebra $A$.

First form the space of finite sums of elementary tensors
$f=\sum_jf_j\otimes x_j$ where the $f_j$ are compactly supported
smooth functions on $[0,\infty)$ and the $x_j\in X$. Then complete this space
using the $C^*$-module norm coming from the inner
product
$$ (f|g)_{L^2([0,\infty))\otimes X}
=\sum_{i,j}\int_0^\infty \bar{f_i}(t)g_j(t)dt \,(x_i|y_j)_X,$$ which
for convenience we write as
$$( f|g)_{L^2([0,\infty))\otimes X}:=\int_0^\infty(f_t|g_t)_Xdt.$$
This module is the external tensor product $L^2([0,\infty))\otimes
X$. It carries an obvious left action of $M(F,A)$. We caution the
reader that it is not clear that the {\em completion} of the space of
finite sums of elementary tensors is a function space.
Discussion of this matter and the proof that the next definition does in fact
provide an unbounded Kasparov module can be found in \cite{CPR1}.
\begin{defn}[\cite{CPR1}]\label{dfn:APS-module}
Assume the SSA is satisfied. Define a graded unbounded equivariant
Kasparov $M(F,A)$-$F$-module by
$$(\hat{X},\hat\D)=\left(\bca
L^2([0,\infty))\otimes X\\L^2([0,\infty))\otimes X\oplus
\Phi_0X\eca,
\bma 0 & -\p_t+\D\\
\p_t+\D & 0\ema\right),$$ where we use APS boundary conditions in the
sense that we take the initial domain of $\hat\D$ to be the finite
linear span of elementary tensors $\xi$ such that $\xi\in \hat{X}$
and $ \hat\D\xi\in\hat{X}$ with $\ P\xi_1(0)=0,\ (1-P)\xi_2(0)=0$,
where $P=\chi_{[0,+\infty)}(\D)=\sum_{k\ge0}\Phi_k$ is the
non-negative spectral projection of $\D$.
\end{defn}

{\bf Remark}. The additional copy of $\Phi_0X$ (which has as inner
product the restriction of the inner product on $X$) allows us to
use  extended $L^2$ functions as in \cite[pp 58-60]{APS1}. These are
defined by considering functions $f$ that are finite sums of
elementary tensors $\sum_j f_j\otimes x_j$ where the $f_j$ are
functions on $[0,\infty)$ with a limit $f_j(\infty)$ as
$t\to\infty$. Then $f$ has a limit at infinity and we restrict our
attention to those $f$  such that $f-f(\infty)$ is in
$L^2([0,\infty))\otimes X$ and $\D f(\infty)=0$. The inner product
in the second component
is then
$$( f|f)=\int_0^\infty (f(t)-f(\infty)|f(t)-f(\infty))_Xdt
+(f(\infty)|f(\infty))_X.$$

The Kasparov module $(\hat X,\hat \D)$ is equivariant using the circle
action on $A$, which is trivial in the `$\Rl$'-direction. It thus
defines an element of $KK_0^\T(M(F,A),F)$. By pairing it with elements
of $K_0^\T(M(F,A))$ we get a homomorphism $\Ind_{\hat \D}\colon
K_0^\T(M(F,A))\to K_0^\T(F)$.

\begin{thm}[\cite{CPR1}]\label{mainequivariantresult}
Assume the SSA is satisfied. Let $U\colon \T\to\B(\HH)$ be a finite
dimensional unitary representation and $v\in A^\sim\otimes \B(\HH)$ a
$\s\otimes\Ad U$ invariant partial isometry with $v^*v$ and $vv^*$
projections in $F^\sim\otimes \B(\HH)$. Assume that the $\s\otimes
\iota$ homogeneous components of $v$ are partial isometries. Then we
have
\begin{align*}
&\Ind_{\hat\D}\left([p_v]-\left[\bma 1 & 0\\ 0 & 0\ema\right]\right)\\
&=-\Ind\left((P\otimes1)v(P\otimes1)\colon
v^*v(P\otimes1)X\otimes\HH\to vv^*(P\otimes1)X\otimes\HH\right)\in K_0^\T(F).
\end{align*}
\end{thm}

The proof is exactly the same as the non-equivariant result of
\cite{CPR1}, except that one must check that the kernel and cokernel
projections are indeed invariant, which is immediate from the
equivariance of the Kasparov module.

{\bf Remarks.} 

(i) In \cite{CPR1} we could not state the (nonequivariant
version of the) above theorem for every element in $K_0(M(F,A))$, but
only those with particular commutation relations with spectral
projections of $\D$. The additional assumption on $v$ in the above
formulation is enough to get those relations satisfied. Indeed, if
$v\in A^\sim\otimes \B(\HH)$ is homogenous then
$v[\D\otimes1,v^*]=kvv^*$ for some $k\in\Z$, and this commutes with
$\D\otimes1$. 

(ii) One may also try to describe the class $[\hat\D]$ in the following way.
First realise the class $[\D]$ as an extension
\begin{equation*}
0\to \K\otimes F\to E\stackrel{\stackrel{\rho}{\leftarrow}}{\to} A\to 0
\end{equation*}
with $\rho$ the completely positive splitting given by $a\to PaP$, 
for $a\in A$, and $P=\chi_{[0,\infty)}(\D)$. As $P$ commutes with $F$, 
$\rho$ is an injective homomorphism when restricted to $F$, and so 
gives us a copy of $F$ inside $E$. 
From this we may deduce the exactness 
of the sequence
$$
0\to \K\otimes \CS F \to M(F,E) \stackrel{\stackrel{\tilde\rho}{\leftarrow}}{\to} M(F,A) \to 0,
$$
where $M(F,E),\,M(F,A)$ denote the mapping cones of the respective inclusions. Corresponding 
to this extension is a class 
$$[\tilde{\D}]\in KK_1^\T(M(F,A),\CS F)=
KK_1^\T(\CS M(F,A),F)=KK_0^\T(M(F,A),F).$$ 
There is some evidence that the class $[\tilde{\D}]$ coincides with the class of $[\hat\D]$.

The theorem gives us two important tools. The first is that the pairing
of $(\hat{X},\hat{\D})$ is given by the Kasparov product and so enjoys
all the usual functorial properties. The second is that we can compute
the index pairing of the theorem by considering Toeplitz type operators
$(P\otimes 1)v(P\otimes1)$, for which the computation is much simpler.
These tools use only the circle action. Next we 
exploit the KMS weight $\phi$.

\section{The equivariant spectral flow}
\label{sec:spec-trip-index}

\subsection{The induced trace}\label{ss:trace}

A KMS
weight provides some analytic tools that we now explain.

\begin{defn}
A weight $\phi$ on a $C^*$-algebra $A$ is $(\s,\beta)$-KMS weight
(KMS$_\beta$ weight for short) if $\phi$ is a semifinite, norm lower
semicontinuous, $\s$-invariant weight such that
$\phi(aa^*)=\phi(\s_{i\beta/2}(a)^*\s_{i\beta/2}(a))$ for all $a\in
{\rm dom}(\s_{i\beta/2}).$
\end{defn}

Here ${\rm dom}(\s_{i\beta/2})$ consists of all elements $a\in A$ such
that $t\mapsto\s_t(a)$ extends to a continuous function from $0\leq
\Im(t)\leq \beta/2$ which is analytic in the open strip.
We will assume throughout the rest of the paper that $\phi$ is a
faithful KMS$_\beta$ weight on $A$.
Introduce the notation
$$ \mbox{dom}(\phi)_+=\{a\in A_+:\phi(a)<\infty\},\ \ \
\mbox{dom}(\phi)^{1/2} =\{a\in A: a^*a\in  \mbox{dom}(\phi)_+\},$$
$$  \mbox{dom}(\phi)=\mbox{span}\{ \mbox{dom}(\phi)_+\}=
( \mbox{dom}(\phi)^{1/2})^* \mbox{dom}(\phi)^{1/2},$$ and extend $\phi$
to a linear functional on $\mbox{dom}(\phi)$. Recall that we defined a
conditional expectation $\Phi\colon A\to F$. We let $\tau$ be the
faithful norm lower semicontinuous semifinite trace on $F$ given
by~$\phi|_F$. Then $\phi=\tau\circ\Phi$, as $\phi$ is assumed to be
$\sigma$-invariant.

The GNS construction yields a Hilbert space $\HH:=\HH_\phi$, and a map
$\Lambda\colon \mbox{dom}(\phi)^{1/2}\to \HH$ with dense image and $\la
\Lambda(a),\Lambda(b)\ra=\phi(a^*b)$, where $\la\cdot,\cdot\ra$ is the
inner product. In fact, $\Lambda(\mbox{dom}(\phi)^{1/2}\cap
(\mbox{dom}(\phi)^{1/2})^*)$ is a left Hilbert algebra. The algebra $A$
is represented on $\HH$ as left multiplication operators,
$a\Lambda(b)=\Lambda(ab)$, and the weight $\phi$ extends to a normal
semifinite faithful weight on the von Neumann algebra $\pi(A)''$. We
have $\sigma_t^\phi=\sigma_{-\beta t}$ on $A$, \cite{Ku}.


We now construct a semifinite von Neumann algebra from a given faithful
KMS state or weight $\phi$ on~$A$. We need results from \cite{LN} at
this point. Namely consider the space $\HH_\tau$ of the
GNS-representation of the trace $\tau$ on $F$; then $\HH$ can be
identified with $X\otimes_F\HH_\tau$ via the map
$a\otimes\Lambda_\tau(f)\mapsto\Lambda_\phi(af)$. It follows that the
action of $A$ on $\HH$ extends to a representation of $\End_F(X)$ on
$\HH$.

\begin{lemma}\label{spatial trace}
We let $\cn=\End(X)''\subset \B(\HH)$, then there is a faithful normal
semifinite trace ${\rm Tr}_\phi$ on~$\cn$ such that ${\rm
Tr}_\phi(\Theta_{\xi,\xi})=\tau((\xi|\xi)_R)$ for all $\xi\in X$.
\end{lemma}

\begin{proof}
Consider first the case when $X\cong H\otimes F$ as a right Hilbert
$F$-module, where $H$ is a Hilbert space. Then $\HH\cong
H\otimes\HH_\tau$, $\cn\cong \B(H)\bar\otimes F''$ and the trace~${\rm
Tr}_\phi$ is simply ${\rm Tr}\otimes\tau$. We can then conclude that
${\rm Tr}_\phi$ exists if $X$ is only a direct summand of $H\otimes F$.
This is the case when $X$ is countably generated (in particular, when
$A$ is separable) by Kasparov's stabilization theorem, with
$H=\ell^2(\mathbb N)$. In general to construct ${\rm Tr}_\phi$ we can
argue as follows.

The commutant of $\cn$ can be identified with the commutant of $F$ in
$\B(\HH_\tau)$, that is, with the von Neumann algebra generated by
elements $f\in F$ acting on the right. To put it differently,
\begin{equation} \label{com}
\cn'=(JFJ)'',
\end{equation}
where $J$ is the modular conjugation defined by $\phi$ and $F$ acts
on the left. Define a trace $\tau'$ on $(JFJ)''$ by
$\tau'(Jf^*J)=\tau(f)$. At this moment we need to recall the notion
of spatial derivative, see~\cite{S}.

Assume we are given faithful normal semifinite weights $\psi$ on
$\cn$ and $\rho$ on $\cn'$. A vector $\xi\in\HH$ is called
$\rho$-bounded if the map $\Lambda_{\rho}(x)\mapsto x\xi$,
$x\in\mbox{dom}(\rho)^{1/2}$, extends to a bounded map
$R_\xi\colon\HH_{\rho}\to\HH$. As $R_\xi$ is an $\cn'$-module map,
the operator $R_\xi R_\xi^*$ belongs to $\cn$. The quadratic form
$$
\{\xi\in\HH\mid\xi\ \hbox{is}\ \ \rho\hbox{-bounded},\ \psi(R_\xi
R_\xi^*)<\infty\}\ni\xi\mapsto q(\xi):=\psi(R_\xi R_\xi^*)
$$
is closable and hence defines a positive self-adjoint operator
$\Delta(\psi/\rho)$ such that
$q(\xi)=\|\Delta(\psi/\rho)^{1/2}\xi\|^2$. The main property of
spatial derivatives is that for any fixed $\rho$ the map
$\psi\mapsto\Delta(\psi/\rho)$ gives a one-to-one correspondence
between faithful normal semifinite weights $\psi$ on $\cn$ and
nonsingular positive self-adjoint operators~$\Delta$ such that
$\Delta^{it}x\Delta^{-it}=\sigma^{\rho}_{-t}(x)$ for $x\in\cn'$.

The spatial derivative now gives us the definition of a trace
$\mbox{Tr}_\phi$ on $\cn=\End(X)''$ by requiring $\Delta({\rm
Tr}_\phi/\tau')=1$. It is not difficult to check, see
\cite[Section~3]{LN}, that for $\xi\in X$ we indeed have ${\rm
Tr}_\phi(\Theta_{\xi,\xi})=\tau((\xi|\xi)_R)$.
\end{proof}

The restriction of ${\rm Tr}_\phi$
to $\End_F(X)$ is a strictly lower semicontinuous strictly
semifinite trace, see e.g. \cite[Section 3]{LN}. In addition we notice that
as $\Theta_{x,x}=xx^*\Phi_0$ and $\tau(xx^*)=\tau(x^*x)$ for
$x\in F$, we can conclude that
\begin{equation} \label{tr}
{\rm Tr}_\phi(f\Phi_0)=\tau(f)
\end{equation}
for $f\in F_+$. Identities (\ref{com})-(\ref{tr}) mean that $\cn$ is
being given by the basic von Neumann algebra  construction associated
with the conditional expectation $\Phi\colon A''\to F''$, while ${\rm
Tr}_\phi$ is the canonical trace on $\cn$ defined by the trace $\tau$
on $F''$, \cite{Po}.

\begin{lemma}\label{lemma:not-too-big} Let $A,\,\sigma,\,\phi$, $F=A^\sigma$ 
be as above.
For all $f\in F$, $f\geq 0$ and $k\in\Z$, $k\neq 0$, we have
$$ \Tr_\phi(f\Phi_k)\leq e^{k\beta} \tau(f),$$
and equality holds if $A$ has full spectral subspaces.
\end{lemma}

\begin{proof}
Consider first $f=xx^*$, $x\in A_k$. Then $f\Phi_k=\Theta_{x,x}$ and
hence
$$
\Tr_\phi(f\Phi_k)=\phi(x^*x)=e^{k\beta}\phi(xx^*)=e^{k\beta}\tau(f).
$$
Therefore $\Tr_\phi(f\Phi_k)=e^{k\beta}\tau(f)$ if $f$ is a finite sum
of elements of the form $xx^*$, $x\in A_k$. Since both
$\Tr_\phi(\cdot\,\Phi_k)$ and $\tau$ are lower semicontinuous traces on
$F$, we conclude that, as $A_kA_k^*$ is a dense ideal in~$F_k$,
$\Tr_\phi(f\Phi_k)=\tau(f)$ for any $f\in F_k$, $f\ge0$. Thus if
$F_k=F$ for all $k\in\Z$ we get equality for all $f\ge0$ and $k$.

In the more general situation consider the ideal
$F_k=\overline{A_kA_k^*}$ in $F$. Choose an approximate unit
$\{\psi_\lambda\}_\lambda$ for $F_k$. Since $A_kA_k^*A_k$ is dense in
$A_k$, we have $\psi_\lambda x\to x$ for any $x\in X_k$. Hence
$\psi_\lambda f\psi_\lambda$ converges strongly to the action of~$f$ on
$X_k$ for any $f\in F$. Since $\Tr_\phi$ is strictly lower
semicontinuous, for $f\ge0$ we therefore get
$$
\Tr_\phi(f\Phi_k)\le\liminf_\lambda \Tr_\phi(\psi_\lambda f\psi_\lambda\Phi_k)
=\liminf_\lambda e^{k\beta}\tau(\psi_\lambda f\psi_\lambda)=
\liminf_\lambda e^{k\beta}\tau(f^{1/2}\psi_\lambda^2 f^{1/2})
\le e^{k\beta}\tau(f).
$$
\end{proof}

\noindent {\bf Remark}. As the proof shows we do not need compactness
of the spectral projections $\Phi_k$, only the strong convergence of
$\psi_\lambda f\psi_\lambda$ to $f$ on $X_k$ and strict lower
semicontinuity of $\Tr_\phi$. If the SSA holds then $\psi_\lambda
f\psi_\lambda\to f$ on $X_k$ in norm.

\subsection {The spectral flow}

Our reference for Breuer-Fredholm theory and semifinite spectral flow is \cite{BCPRSW}. We
recall from Section 6 of that paper
that if $\cn$ is a semifinite von Neumann algebra with faithful
normal semifinite trace $\tau$ and $\D_1$, $\D_2$ are closed
self-adjoint operators affiliated with $\cn$ which differ by a bounded operator and
whose spectral
projections $P_1=\chi_{[0,+\infty)}(\D_1)$ and
$P_2=\chi_{[0,+\infty)}(\D_2)$ are such that the operator $P_1P_2\in P_1\cn P_2$ is
Breuer-Fredholm, then the spectral flow is defined by
$$
sf(\D_1,\D_2)=\Ind_\tau(P_1P_2).
$$
In the case when $P_1$ and $P_2$ are finite we clearly have
$
sf(\D_1,\D_2)=\tau(P_2)-\tau(P_1).
$

Now let $A$, $\HH$, $\cn$ be as in the previous Subsection. The
unbounded operator $\D$ on $X$, introduced in Subsection~\ref{ss:KK},
extends to a closed self-adjoint operator on $\HH$, which we still
denote by $\D$. Put $\s_t(x)=e^{it\D}xe^{-it\D}$ for $x\in\cn$. The
action of $\T$ on $X$ extends to a unitary representation of $\T$ on
$\HH$, namely, $t\mapsto e^{it\D}$. We want to define a map from
$K_0^\T(M(F,A),F)$ to the representation ring of the circle which we
will call the equivariant spectral flow. Roughly speaking it will
compute the spectral flow between the operators $vv^*(\D\otimes1)$ and
$v(\D\otimes1)v^*$ on invariant subspaces for the $\T$-action. However,
if $\phi$ is a weight, even the restriction of the above operators to
an invariant subspace may not be enough to get a well-defined spectral
flow. So we have to pay attention to domain issues.

\begin{lemma}\label{lm:doms} Let $A$ be a $C^*$-algebra and $\phi$ a weight on $A$ as above. 
For every $n\in\mathbb N$ the dense subalgebra $\Mn(\dom(\phi)^\sim)$
of $\Mn(A^\sim)$ is closed under the holomorphic functional calculus.
\end{lemma}

\begin{proof}
In order to prove the Lemma it suffices to show that if $\Gamma$ is a
closed smooth curve in $\C$, $a\in\Mn(\dom(\phi)^\sim)$ has spectrum
(as an element of $\Mn(A^\sim)$) which does not intersect $\Gamma$ then
for any continuous function $f\colon\Gamma\to\C$ the integral
$\int_\Gamma f(z)(z-a)^{-1}|dz|$ defines an element in
$\Mn(\dom(\phi)^\sim)$. Let $b\in\Mn(\C)$ be such that
$c:=a-b\in\Mn(\dom(\phi))$. Then the spectrum of $b$ is contained in
that of~$a$, so we just have to show that
$$
\int_\Gamma f(z)\big((z-a)^{-1}-\big(z-b)^{-1}\big)|dz|\in\Mn(\dom(\phi)).
$$

For this observe that if $[0,1]\ni t\mapsto a_t,b_t\in\dom(\phi)^{1/2}$
are two continuous maps such that the functions $t\mapsto
\phi(a_t^*a_t),\phi(b_t^*b_t)$ are bounded, then $\int^1_0a_t^*b_tdt\in
\dom(\phi)$. Indeed, by the polarization identity it is enough to
consider the case $a_t=b_t$, and then the claim follows from lower
semicontinuity. Observe also that for $d\in\Mn(A)$ we have
$d\in\Mn(\dom(\phi)^{1/2})$ if and only if
$(\phi\otimes\Tr)(d^*d)<\infty$. Denote by $\mathcal G$ the class of
continuous functions $\Gamma\to\Mn(A)$ which are finite sums of
functions of the form $z\mapsto d_z^*e_z$ such that $d_z,e_z\in\Mn(A)$
depend continuously on $z$ and the functions
$z\mapsto(\phi\otimes\Tr)(d_z^*d_z),(\phi\otimes\Tr)(e_z^*e_z)$ are
bounded. The integral of any function in $\mathcal G$ defines an
element of $\Mn(\dom(\phi))$. Therefore it suffices to show that the
function $z\mapsto (z-a)^{-1}-\big(z-b)^{-1}$ is in $\mathcal G$.

The class $\mathcal G$ contains constant $\Mn(\dom(\phi))$-valued
functions and is stable under multiplication (from either side) by
continuous $\Mn(\C)$-valued functions. In particular, the function
$$
z\mapsto c_z:=c(z-b)^{-1}
$$
is in $\mathcal G$. Furthermore, if $f_1,f_2\in\mathcal G$ and
$f\colon\Gamma\to\Mn(A^\sim)$ is continuous then $f_1ff_2\in\mathcal
G$. The identities
$$
(z-a)^{-1}-\big(z-b)^{-1}=(z-b)^{-1}\big((1-c_z)^{-1}-1\big)
=(z-b)^{-1}\big(c_z+c_z(1-c_z)^{-1}c_z\big)
$$
show then that $(z-a)^{-1}-\big(z-b)^{-1}$ is indeed in $\mathcal G$.
\end{proof}

Observe next that if $U\colon\T\to \B(\HH_U)$ is a finite dimensional
unitary representation then any $\s\otimes\Ad U$ invariant element is a
finite sum of homogeneous components with respect to $\s\otimes\iota$.
So to deal with equivariant $K$-theory of $A$ the following algebra is
enough.

\begin{defn}
Denote by $\A$ the algebra consisting of finite sums of
$\s$-homogeneous elements in the domain $\dom(\phi)$ of $\phi$. We also
put $\F=\A\cap F=\dom(\tau)$.
\end{defn}

We next turn to equivariant $K$-theory of the mapping cone.

\begin{lemma} \label{lem:rep}
Every class in $K^\T_0(M(F,A))$ has a representative $v$ such that
$v\in(\A^\sim\otimes \B(\HH_U))^{\s\otimes\Ad U}$, $vv^*$ and $v^*v$
are in $\F^\sim\otimes \B(\HH_U)$, and $vv^*=v^*v$ modulo $\F\otimes
\B(\HH_U)$, where $U\colon \T\to \B(\HH_U)$ is a finite dimensional
unitary representation.
\end{lemma}

\begin{proof}
By Lemma~\ref{lm:doms} and Putnam's description of $K$-theory of the
mapping cone~\cite{Put} we first conclude that every class has a
representative $v$ such that $v\in(\A^\sim\otimes
\B(\HH_U))^{\s\otimes\Ad U}$  and $vv^*,v^*v\in\F^\sim\otimes
\B(\HH_U)$. The images of the projections $vv^*$ and $v^*v$ in
$\B(\HH_U)^{\Ad U}$ under the isomorphism $$(\F^\sim\otimes
\B(\HH_U))/(\F\otimes \B(\HH_U))\cong \B(\HH_U)$$ are equivalent, so
there exists a $\s\otimes\Ad U$-invariant unitary $u\in\F^\sim\otimes
\B(\HH_U)$ such that $uvv^*u=v^*v$ modulo $\F\otimes \B(\HH_U)$. It
remains to recall~\cite{Put} that the classes of $v$ and $uv$ coincide,
so that $uv$ is the required representative.
\end{proof}

We are now ready to define the equivariant spectral flow. First
consider homogeneous subspaces. Let $U$ and $v$ be as in the above Lemma.
Let $\Psi_n\colon \HH_U\to\HH_U$, resp.
$Q_n\colon\HH\otimes\HH_U\to\HH\otimes\HH_U$, be the projection onto
the $\chi^n$-homogeneous  component, so that
$Q_n=\sum_k\Phi_{n-k}\otimes\Psi_k$. We then define
$$
sf_n(v)=(\Tr_\phi\otimes\Tr)((v^*v-vv^*)Q_n(P\otimes1))\in\Rl,
$$
where $P=\chi_{[0,+\infty)}(\D)=\sum_{k\ge0}\Phi_k$. Observe that this
quantity is finite by Lemma~\ref{lemma:not-too-big}, since
$v^*v-vv^*\in\F\otimes \B(\HH_U)$ by assumption and $\dim\HH_U$ is
finite.

\begin{lemma} \label{lem:isoflow}
The value $sf_n(v)$ depends only on the class of $v$ in
$K^\T_0(M(F,A))$.
\end{lemma}

\begin{proof}
Denote by $\tilde\tau$ the normal semifinite trace
$(\Tr_\phi\otimes\Tr)(\cdot\,Q_n(P\otimes1))$ on $\cn^\s\otimes
\B(\HH_U)^{\Ad U}$. It suffices to show that if $v_t\in \A^\sim\otimes
\B(\HH_U)$, $t\in(0,1)$, is a continuous path of partial isometries
satisfying the properties in the formulation of
Lemma~\ref{lemma:not-too-big}, then $\tilde
\tau(v_0v_0^*-v_0^*v_0)=\tilde\tau(v_1v_1^*-v_1^*v_1)$. Since the
images of the projections $v_tv_t^*$ in $\B(\HH_U)$ are equivalent, we
can find a continuous path of $\s\otimes\Ad U$-invariant unitaries
$u_t\in \F^\sim\otimes \B(\HH_U)$ such that $v_0v_0^*=u_tv_tv_t^*u_t$
modulo $\F\otimes \B(\HH_U)$. Replacing $v_t$ by $u_tv_tu_t^*$ we may
therefore assume that the projections $v_tv_t^*$ coincide modulo
$\F\otimes \B(\HH_U)$. Then it suffices to check that if
$p_t\in\F^\sim\otimes \B(\HH_U)^{\Ad U}$, $t\in(0,1)$, is a continuous
path of projections which coincide modulo $\F\otimes \B(\HH_U)$ then
$\tilde\tau(p_0-p_1)=0$. We may assume that $\|p_0-p_1\|<1$. Consider
the invertible element $w=p_0p_1+(1-p_0)(1-p_1)\in
F^\sim\otimes\B(\HH_U)^{\Ad U}$. Then $p_0=wp_1w^{-1}$ and
$w-1\in\F\otimes\B(\HH_U)^{\Ad U}$. Hence
$$
\tilde\tau(p_0-p_1)=\tilde\tau((w-1)p_1w^{-1})+\tilde\tau(p_1(w^{-1}-1))
=\tilde\tau(p_1w^{-1}(w-1))+\tilde\tau(p_1(w^{-1}-1))=0.
$$
\end{proof}

Thus we get a well-defined map $sf_n\colon K^\T_0(M(F,A))\to\Rl$.

\begin{lemma}
For every $[v]\in K^\T_0(M(F,A))$ we have $sf_n([v])=0$ for all but a
finite number of $n\in\Z$.
\end{lemma}

\begin{proof}
In the notation before Lemma~\ref{lem:isoflow}, we have $\Psi_k=0$ for
$|k|$ large enough. It follows that $Q_n(P\otimes1)=0$ for all $n\in\Z$
small enough and $Q_n(P\otimes1)=Q_n$ for $n$ sufficiently large.
Therefore it suffices to check that for the normal semifinite trace
$\tilde\tau=(\Tr_\phi\otimes\Tr)(\cdot\,Q_n)$ on $(\cn\otimes
\B(\HH_U))^{\s\otimes\Ad U}$ we have $\tilde\tau(vv^*-v^*v)=0$. This is
true since by assumption $v-w$ belongs to the domain of $\tilde\tau$
for an element $w\in\C\otimes\B(\HH_U)^{\Ad U}$ such that $w^*w=ww^*$.
\end{proof}


\begin{defn}\label{dfn:equiflow}
The $\T$-equivariant spectral flow is the map $sf\colon
K^\T_0(M(F,A))\to\Rl[\chi,\chi^{-1}]$ defined by
$$
sf([v])=\sum_{n\in\Z}sf_n([v])\chi^n.
$$
\end{defn}

By Theorem~\ref{mainequivariantresult} and definition of the induced
trace we may conclude that if the SSA is satisfied then the equivariant
spectral flow coincides with the composition
$$
K^\T_0(M(F,A))\xrightarrow{-\Ind_{\hat\D}}K_0^\T(F)=K_0(F)[\chi,\chi^{-1}]
\xrightarrow{\tau_*}\Rl[\chi,\chi^{-1}],
$$
at least on the elements represented by $\s\otimes \iota$ homogeneous
isometries $v$; here $\tau_*$ denotes the homomorphism $K_0(F)\to\Rl$
defined by the trace $\tau$. We shall return to this in more detail in
the next Section.

\section{Modular index pairing}

\subsection{Modular $K_1$}
In the previous Section we defined an equivariant spectral flow which
assigns to an invariant partial isometry $v\in A^\sim\otimes \B(\HH_U)$
a Laurent polynomial in $\chi$. Being evaluated at $\chi=1$ this
polynomial gives a suitably defined spectral flow from
$vv^*(\D\otimes1)$ to $v(\D\otimes1)v^*$ with respect to ${\rm
Tr}_\phi\otimes\Tr$. We would like to obtain an analytic formula for
this spectral flow. Such formulas are available under certain
summability assumptions, but as Lemma~\ref{lemma:not-too-big} shows,
even when $\phi$ is a state, the operator $|\D|^{-p}$ is not summable
in general for any $p>0$. The same Lemma suggests, however, that to
improve summability it would suffice to assign the weight $e^{-n\beta}$
to every projection $\Phi_n$. Effectively this means that we evaluate
the equivariant spectral flow at $\chi=e^{-\beta}$. This was done from
a different point of view in~\cite{CPR2}, where notions of a modular
$K_1$ group and a modular pairing were introduced. Our considerations
allow us to relate the results of~\cite{CPR2} to more conventional
constructions.

The following definition is essentially from~\cite{CPR2}, slightly
modified and extended to adapt to our current considerations.

\begin{defn}
A partial isometry in $A^\sim$ is modular if $v\s_t(v^*)$ and
$v^*\s_t(v)$ are in $(A^\sim)^\s$ for all $t\in\Rl$. By a modular
partial isometry over $A$ we mean a modular partial isometry in
$\Mn(A^\sim)=A^\sim\otimes\Mn(\C)$ for some $n\in\N$ with respect to
the action $\s\otimes\iota$.
\end{defn}

In~\cite{CPR2} only modular unitaries were considered. Observe that
every modular partial isometry $v$ over~$A$ defines a modular unitary
by
$$
u_v=\bma 1-v^*v & v^*\\ v & 1-vv^*\ema.
$$
Define the modular $K_1$ group as follows.

\begin{defn} Let $K_1(A,\s)$ be the abelian group with one
generator $[v]$ for each partial isometry $v$ over~$A$ satisfying the
modular condition and with the following relations: \bean 1)&& [v]=0\
\mbox{if}\ v\ \mbox{is over}\ F,\nno 2)&& [v]+[w]=[v\oplus w],\nno 3)&&
\mbox{if }v_t,\ t\in[0,1],\ \mbox{is a continuous path of modular
partial isometries in }\Mn(A^\sim)\nno && \mbox{then}\
[v_0]=[v_1].\eean
\end{defn}

{\bf Remarks}.  It is easy to show that $v\oplus w\sim w\oplus v$,
see~\cite{CPR2}, however the inverse of $[v]$ is not~$[v^*]$
in general. Equivalently, even though $u_v$ is a self-adjoint unitary
and hence is homotopic to the identity, such a homotopy cannot always
be chosen to consist of modular unitaries.

Observe that $\s$-homogeneous partial isometries are modular. It turns
out that they generate the whole group $K_1(A,\sigma)$. We need some
preparation to prove this.

\begin{lemma} \label{modu0}
A unitary $u\in A^\sim$ is modular if and only if there exists
a self-adjoint element $a\in F^\sim$ such that
$
uau^*\in F^\sim$ and $\sigma_t(u)=ue^{ita}$ for
$t\in\Rl.
$
\end{lemma}

\begin{proof}
Put $u_t=u^*\sigma_t(u)$. Then
$$
u_{t+s}=u^*\sigma_{t+s}(u)=u^*\sigma_t(u)\sigma_t(u^*\sigma_s(u))
=u_tu_s.
$$
Thus $\{u_t\}_t$ is a norm-continuous one-parameter group of unitary
operators in $F^\sim$. Hence there exists a self-adjoint element $a\in
F^\sim$ such that $u_t=e^{ita}$. Therefore
$$
\sigma_t(u)=ue^{ita}=e^{ituau^*}u.
$$
Since $u$ is modular, the second equality implies that $uau^*\in
F^\sim$. The converse is obvious.
\end{proof}


For an element $x\in\Mn(A^\sim)$ we denote by $x_k$ the spectral
component of $x$ with respect to $\s\otimes\iota$, so
$(\s_t\otimes\iota)(x_k)=e^{ikt}x_k$.

\begin{lemma}\label{moduper}
A partial isometry $v\in\Mn(A^\sim)$ is modular if and only if the
elements $v_k$ are partial isometries which are zero for all but a
finite number of $k$'s and the source projections $v_k^*v_k$, $k\in\Z$,
as well as the range projections $v_kv_k^*$, $k\in\Z$, are mutually
orthogonal.
\end{lemma}

\begin{proof}
Consider the modular unitary $u=u_v$. If $\sigma_t(u)=ue^{ita}$ with
$a$ as in Lemma~\ref{modu0} (but now $a\in\Mat_{2n}(F^\sim)$), then
$u=ue^{2\pi i a}$. Hence the spectrum of $a$ is a finite subset of
$\Z$. Let $p_k$ be the spectral projection of $a$ corresponding to
$k\in\Z$. Then $u_k=up_k$, and hence the partial isometries $u_k$ have
mutually orthogonal sources and ranges. We clearly have
$$
u_0=\bma 1-v^*v & v^*_0\\ v_0 & 1-vv^*\ema,\ \
u_k=\bma 0 & v^*_{-k}\\ v_k & 0\ema\ \ \hbox{for}\ \ k\ne0.
$$
This implies that $v_k=0$ for all but a finite number of $k$, and the
elements $v_k$, $k\ne0$, are partial isometries with mutually
orthogonal sources and ranges. Consider $w=\sum_{k\ne0}v_k$. Then $w$ is a partial
isometry and $ww^*=\sum_{k\ne0}v_kv_k^*$, $w^*w=\sum_{k\ne0}v_k^*v_k$.
Since
$$
v^*v=v_0^*v_0+w^*w+\sum_{k\ne0}(v_0^*v_k+v_k^*v_0)
$$
is invariant, we get $v^*v=v^*_0v_0+w^*w$. Since $v^*v$ and $w^*w$ are
projections, it follows that $v_0^*v_0$ is a projection orthogonal to
$w^*w$. In other words, $v_0$ is a partial isometry with the source
projection orthogonal to $v_k^*v_k$, $k\ne0$. Similarly one checks that
the projections $v_0v_0^*$ and $v_kv_k^*$, $k\ne0$, are orthogonal.

The converse statement is straightforward.
\end{proof}

\begin{cor}
The group $K_1(A,\sigma)$ is generated by the classes of homogeneous
partial isometries.
\end{cor}

\begin{proof}
It suffices to observe that if $v$ and $w$ are modular partial
isometries such that $v^*vw^*w=vv^*ww^*=0$, then $[v+w]=[v]+[w]$.
Indeed, if $R_t=\bma \cos t & \sin t\\ -\sin t & \cos t\ema$, then
$$
v_t=\left(\bma 1-ww^* & 0\\ 0 & 1-ww^*\ema +R_tww^*\right)
\bma v+w & 0\\ 0 & 0\ema
\left(\bma 1-w^*w & 0\\ 0 & 1-w^*w\ema +R_{-t}w^*w\right),
$$
$0\le t\le \pi/2$, is a modular homotopy from $\bma v+w & 0\\ 0 & 0\ema$ to
$\bma v & 0\\ 0 & w\ema$.
\end{proof}

We next want to relate the group $K_1(A,\s)$ to $K_0^\T(M(F,A))$.

Recall that if $\K$ is a finite dimensional Hilbert space considered
with the trivial $\T$-module structure, we denote by $\K[n]$ the same
space with the representation $t\mapsto e^{int}$. Assume $v\in
A^\sim\otimes \B(\K)$ is a partial isometry such that $v\in
A_n^\sim\otimes \B(\K)$, so $(\s_t\otimes\iota)(v)=e^{int}v$, then the
partial isometry
$$
w_v=\begin{pmatrix}0 & v\\ 0& 0\end{pmatrix}\in A^\sim\otimes \B(\K\oplus \K[n])
$$
is $\T$-invariant, so it defines an element of $K^\T_0(M(F,A))$.
Sometimes we shall denote the class $[w_v]\in K_0^\T(M(F,A))$ by
$\ll\!\! v\!\!\gg$. Note that if $n=0$ and so $v$ itself represents an
element of $K^\T_0(M(F,A))$, there is no ambiguity in this notation as
$$
\begin{pmatrix}
0 & v\\ 0 & 0
\end{pmatrix}
\ \ \text{is homotopic to}\ \
\begin{pmatrix}
v & 0\\ 0 & 0
\end{pmatrix},
$$
and moreover, the class of $v$ can easily be shown to be zero, see
\cite[Lemma~2.2(v)]{Put}.

\begin{prop}\label{modular-homo}
The map
$$
v\mapsto \sum_k\ll\!\! v_k\!\!\gg\in K^\T_0(M(F,A))
$$
defined on modular partial isometries gives a homomorphism $T\colon
K_1(A,\sigma)\to K^\T_0(M(F,A))$.
\end{prop}

\begin{proof}
Since homotopic elements have homotopic spectral components, it is
clear that the images of homotopic modular partial isometries coincide.
It follows that we have a well-defined homomorphism $T\colon
K_1(A,\sigma)\to K^\T_0(M(F,A))$; in fact, for each $k$ the
map $[u]\mapsto\ll\!\! u_k\!\!\gg$ is a homomorphism.
\end{proof}

This homomorphism makes it clear why $-[v]\ne[v^*]$ in $K_1(A,\s)$ in
general. Indeed, observe first that in the group $K_0^\T(M(F,A))$ we do
have $-[w]=[w^*]$, basically because $u_w$ is an invariant self-adjoint
unitary, hence there is a homotopy from $u_w$ to $1$ consisting of
invariant unitaries. In particular, for homogeneous $v$ as above we
have $-[w_v]=[w_v^*]$. The class $w_v^*$ is represented by
$$
\bma 0 & v^*\\ 0 & 0 \ema\in A^\sim\otimes \B(\K[n]\oplus \K),\ \
\hbox{while}\ \
w_{v^*}=\bma 0 & v^*\\ 0 & 0 \ema\in A^\sim\otimes \B(\K\oplus \K[-n]).
$$
Therefore $[w_v^*]=\chi^n[w_{v^*}]$. In other words, $-\ll\!\!
v\!\!\gg=\chi^n\ll\!\! v^*\!\!\gg$, so that $T(-[v])=\chi^nT([v^*])$.
Equivalently, we have
$$
T([u_v])=\ll\!\!v\!\!\gg+\ll\!\!v^*\!\!\gg=(1-\chi^{-n})\ll\!\!v\!\!\gg.
$$

\subsection{Modular index}
Recall that in Subsection~\ref{ss:trace} we constructed a semifinite
von Neumann algebra $\cn=\End(X)''\subset \B(\HH)$, a faithful
semifinite normal trace $\Tr_\phi$ and an operator
$\D=\sum_{k\in\Z}k\Phi_k$ on $\HH$.

We now define a new weight on $\cn$.

\begin{defn}
Consider the operator $e^{-\beta\D}=\sum_{k\in\Z}e^{-k\beta}\Phi_k$.
For $S\in\cn_+$ define
$$\phi_\D(S)={\rm Tr}_\phi(e^{-\beta\D/2} Se^{-\beta\D/2}).$$
\end{defn}

Since $e^{-\beta\D}$ is strictly positive and affiliated to $\cn$,
$\phi_\D$ is a faithful semifinite normal weight. Since
$\mbox{Tr}_\phi$ is a trace, the modular group of $\phi_\D$ is given by
$\s^{\phi_\D}_t(\cdot)=e^{-it\beta\D}\cdot e^{it\beta\D}$. The
restriction of~$\s^{\phi_\D}_t$ to~$A$ coincides with $\s_{-\beta t}$.
While $\phi_\D$ is not a trace on $\cn$, it is clearly a semifinite
normal trace on the invariant subalgebra $\cM:=\cn^\s$. The following
Lemma captures the main reason for defining $\phi_\D$.

\begin{lemma}\label{lemma:not-too-big2}
With $A,\,\sigma,\,\phi$ as above, we 
have $f(1+\D^2)^{-1/2}\in\LL^{(1,\infty)}(\cM,\phi_\D)$ if
$f\in\F=\operatorname{dom}(\phi)\cap F$.
\end{lemma}

\begin{proof}
This follows immediately from Lemma~\ref{lemma:not-too-big}, since
$\phi_\D(f\Phi_k)\le\phi(f)$ for $f\ge0$.
\end{proof}

We will call the data  $(\A,\HH,\D,\cn,\phi_\D)$ the {\bf modular
spectral triple} for $(A,\s,\phi)$. It provides us with a way
to compute the spectral flow from $vv^*\D$ and $v\D v^*$ with
respect to the trace $\phi_\D$ on $\cM$ for appropriate partial
isometries in $\A$. The next Lemma justifies our definition of modular
partial isometries.

\begin{lemma}
Let $v\in\Mn(A^\sim)$ be a partial isometry such that
$vv^*,v^*v\in\Mn(F^\sim)$. Then we have
$v(Q\otimes1)v^*,v^*(Q\otimes1)v\in\Mn(\cM)$ for every spectral
projection $Q$ of $\D$ if and only if $v$ is modular.
\end{lemma}

\begin{proof}
Replacing $v$ by $u_v$ we may assume that $v$ is unitary.
Next, suppose first that $v$ is modular. Write $\tilde\s$ for
$\s_t\otimes\iota$ and $\tilde Q$ for $Q\otimes1$. Since $\cM=\cn^\s$,
we need to show that $v\tilde Q v^*$ is $\tilde\s$-invariant. We have
$$
\tilde\s(vQv^*)=\tilde\s(v)\tilde Q\tilde\s(v^*)=vv^*\tilde\s(v)\tilde Q\tilde\s(v^*)
=v\tilde Qv^*\tilde\s(v)\tilde\s(v^*)=v\tilde Qv^*.
$$
A similar argument shows that $v^*\tilde Qv$ is invariant.

On the other hand, if
$$ v\tilde Qv^*=\tilde\s(v\tilde Qv^*)=\tilde\s(v)\tilde
Q\tilde\s(v^*),$$ then $v^*\tilde \s(v)$ commutes with $\tilde
Q=Q\otimes1$. If this is true for every spectral projection $Q$ of the
generator~$\D$ of $\s$, then $v^*\tilde \s(v)$ is
$(\s\otimes\iota)$-invariant. Similarly $v\tilde\s(v^*)$
is invariant. Hence $v$ is modular.
\end{proof}

Next we show that the spectral flow is indeed well-defined for modular
partial isometries.

\begin{lemma} \label{lemma:msf}
For a modular partial isometry $v\in A^\sim\otimes \B(\K)$ consider the
projections $$P_1=\chi_{[0,+\infty)}(vv^*(\D\otimes1))\ \ \hbox{and}\ \
P_2=\chi_{[0,+\infty)}(v(\D\otimes1)v^*).$$ Then the operator
$P_1P_2\in P_1(\cM\otimes \B(\K))P_2$ is Breuer-Fredholm and
\begin{align*}
&sf_{\phi_\D\otimes\Tr}(vv^*(\D\otimes1),v(\D\otimes1)v^*)\\
&=\sum_{k<0}\sum_{k\le n<0}e^{-\beta n}(\Tr_\phi\otimes\Tr)(v_kv_k^*(\Phi_n\otimes1))
-\sum_{k>0}\sum_{0\le n<k}e^{-\beta n}(\Tr_\phi\otimes\Tr)(v_kv_k^*(\Phi_n\otimes1)).
\end{align*}
\end{lemma}

\begin{proof}
By Lemma~\ref{moduper} the element $v$ is a finite sum of its
homogeneous components $v_k$ which are partial isometries with mutually
orthogonal sources and ranges. The operators $vv^*(\D\otimes1)$ and
$v(\D\otimes1)v^*$ commute with $v_kv_k^*$ and
$$
v_kv_k^*vv^*(\D\otimes1)=v_kv_k^*(\D\otimes1),\ \
v_kv_k^*v(\D\otimes1)v^*=v_k(\D\otimes1)v^*_k.
$$
This shows that without loss of generality we may assume that $v$ is
homogeneous, say $v=v_k$. Furthermore, for $k=0$ the operators
coincide, so we just have to consider the case $k\ne0$.

Let $P=\chi_{[0,+\infty)}(\D)=\sum_{n\ge0}\Phi_n$. Since $vv^*$ and
$v^*v$ commute with $\D$, we have
$$
P_1=1-vv^*+vv^*(P\otimes1)\ \ \hbox{and}\ \ P_2=1-vv^*+v(P\otimes1)v^*.
$$
But using homogeneity we can actually say much more and easily express
these projections in terms of $vv^*$ and $\Phi_n$. Namely, as
$v(\Phi_n\otimes1)=(\Phi_{n+k}\otimes1)v$, we have
\begin{equation}
v(P\otimes1)v^*=\sum_{n\ge k}vv^*(\Phi_n\otimes1).
\label{eq:I-can-count}
\end{equation}
With this information it is easy to show that $P_1P_2$ is
Breuer-Fredholm, since this is implied by $P_1-P_2$ being compact in
$\cM\otimes \B(\K)$. However from Equation \eqref{eq:I-can-count} we
have
$$ P_1-P_2=\sum_{n=0}^{k-1}vv^*(\Phi_n\otimes1),\ k>0,\qquad
P_1-P_2=-\sum_{n=k}^{-1}vv^*(\Phi_n\otimes1),\ k<0.$$
To finish the proof it therefore remains to show that for every $n$ the
projection $vv^*(\Phi_n\otimes1)$ has finite trace with respect to
$\phi_\D\otimes\Tr$. By the same argument as in the proof of
Lemma~\ref{lemma:not-too-big} we have
$$
(\phi_\D\otimes\Tr)(vv^*(\Phi_n\otimes1))=
e^{-\beta n}(\Tr_\phi\otimes\Tr)(vv^*(\Phi_n\otimes1))\le(\tau\otimes\Tr)(vv^*).
$$
Notice now that $v\in A\otimes \B(\K)$, since $v=v_k$ is homogeneous
with $k\ne0$. Hence the projection $vv^*\in F\otimes \B(\K)$ is in the
domain of the semifinite trace $\tau\otimes\Tr$ on $F\otimes \B(\K)$,
since the latter domain contains the Pedersen ideal and, in particular,
every projection.
\end{proof}

Observe that the above proof shows that if $v$ is a modular partial
isometry then $v-v_0\in\A\otimes\B(\K)$.
Notice also that if we have a continuous path of modular partial
isometries then the corresponding projections $P_1$ and $P_2$ also form
norm-continuous paths. It follows that the map
$$
v\mapsto
sf_{\phi_\D\otimes\Tr}(vv^*(\D\otimes1),v(\D\otimes1)v^*)
$$
defines a homomorphism $K_1(A,\s)\to\Rl$; this of course also follows
from the explicit expression for the spectral flow. We call this
homomorphism the {\bf modular index} and denote it by $\Ind_{\phi_\D}$.
The following theorem compares $\Ind_{\phi_\D}$ with the equivariant
spectral flow.

\begin{thm}\label{thm:ind-mod-uni}
The modular index map $\Ind_{\phi_\D}\colon K_1(A,\sigma)\to\Rl$ is the
composition of the maps
$$
K_1(A,\sigma)\xrightarrow{T} K^\T_0(M(F,A))
\xrightarrow{sf}\Rl[\chi,\chi^{-1}]\xrightarrow{\Ev(e^{-\beta})}\Rl,
$$
where $\Ev(e^{-\beta})$ is the evaluation at $\chi=e^{-\beta}$. If the
SSA is satisfied then, equivalently, $\Ind_{\phi_\D}$ is the
composition
$$
K_1(A,\sigma)\xrightarrow{[v]\mapsto\sum_k\ll v_k\gg} K^\T_0(M(F,A))
\xrightarrow{-\Ind_{\hat\D}}K_0^\T(F)
=K_0(F)[\chi,\chi^{-1}]\xrightarrow{\tau_*}
\Rl[\chi,\chi^{-1}]\xrightarrow{\Ev(e^{-\beta})}\Rl.
$$
\end{thm}

\begin{proof}
This is a matter of bookkeeping. Let $v=v_k\in A\otimes \B(\K)$ be a
homogeneous partial isometry, $k\ne0$. Recall that $\ll\!\! v\!\!\gg$
is represented by $w_v=\bma 0&v\\0&0\ema\in \B(\K\oplus\K[k])$. To
compute $sf([w_v])$ first observe that the projection $Q_n$ onto the
$\chi^n$-homogeneous component of $\HH\otimes(\K\oplus\K[k])$ is $\bma
\Phi_n&0\\0&\Phi_{n-k}\ema$. Therefore
\begin{align*}
sf([w_v])&=\sum_n(\Tr_\phi\otimes\Tr)((w_v^*w_v-w_vw_v^*)Q_n(P\otimes1))\chi^n\\
&=\sum_n(\Tr_\phi\otimes\Tr)(v^*v(\Phi_{n-k}P\otimes1)-vv^*(\Phi_nP\otimes1))\chi^n.
\end{align*}
The projection $v^*v(\Phi_{n-k}\otimes1)=v^*(\Phi_n\otimes1)v$ is
equivalent to the projection $vv^*(\Phi_n\otimes1)$. It follows that
the $n$th summand in the above expression is nonzero only when $n-k$
and $n$ have different signs. More precisely, for $k<0$ we get
$$
\sum_{k\le n<0}(\Tr_\phi\otimes\Tr)(v^*v(\Phi_{n-k}\otimes1))\chi^n
=\sum_{k\le n<0}(\Tr_\phi\otimes\Tr)(vv^*(\Phi_n\otimes1))\chi^n,
$$
and for $k>0$ we get
$$
-\sum_{0\le n<k}(\Tr_\phi\otimes\Tr)(vv^*(\Phi_n\otimes1))\chi^n.
$$
For $\chi=e^{-\beta}$ these expressions coincide with those in
Lemma~\ref{lemma:msf}.
\end{proof}

\section{The analytic index from spectral flow}\label{subsec:index-pairing}

\subsection {A spectral flow formula}

Our method of computing numerical invariants from KMS states exploits
semifinite spectral flow and so we need to review the spectral flow
formula of \cite{CP2}. There are two versions of this formula in the
unbounded setting, one for $\theta$-summable spectral triples, and the
other for finitely summable triples. It is the latter that we will want
to use. First we quote  \cite[Corollary 8.11]{CP2}.

\begin{prop} Let $(\A,\HH,\D_0)$ be an odd unbounded 
$\theta$-summable semifinite spectral triple relative to $(\cM,\tau)$,
where $\tau$ is a faithful semifinite normal trace on $\cM$. For any
$\epsilon>0$ we define a one-form $\alpha^\epsilon$ on the affine space
$\cM_0=\D_0+\cM_{sa}$ by
$$\alpha^\epsilon(A)=\sqrt{\frac{\epsilon}{\pi}}\tau(Ae^{-\epsilon\D^2})$$
for $\D\in\cM_0$ and $A\in T_\D(\cM_0)=\cM_{sa}$ (here $T_\D(\cM_0)$ is
the tangent space to $\cM_0$ at $\D$). Then the integral of
$\alpha^\epsilon$ is independent of the piecewise $C^1$ path in $\cM_0$
and if $\{\D_t =\D_a+A_{t}\}_{t\in[a,b]}$ is any piecewise $C^1$ path
in $\cM_0$ joining $\D_a$ and $\D_b$ then
$$
sf(\D_a,\D_b)=\sqrt{\frac{\epsilon}{\pi}}
\int_a^b\tau(\D_t'e^{-\epsilon\D_t^2})dt
+\frac{1}{2}\eta_\epsilon(\D_b)
-\frac{1}{2}\eta_\epsilon(\D_a)
+\frac{1}{2}\tau\left([\ker(\D_b)]-[\ker(\D_a)]\right).$$
Here the truncated eta invariant is given for $\epsilon>0$ by
$$\eta_\epsilon(\D)=
\frac{1}{\sqrt{\pi}}\int_\epsilon^\infty\tau(\D e^{-t\D^2})t^{-1/2}dt.$$
\end{prop}

We want to employ this formula in a finitely summable setting, so we
need to Laplace transform the various terms appearing in the formula.
In fact we were able in \cite{CRT} to translate the formula in
\cite{CP2} for the spectral flow into a residue type formula. The
importance of such a formula lies in the drastic simplification of
computations, since we may throw away terms that are holomorphic in a
neighbourhood of the point where we take a residue.

We introduce the notation
$$ C_r:=\frac{\sqrt{\pi}\Gamma(r-1/2)}{\Gamma(r)}
=\int_{-\infty}^\infty(1+x^2)^{-r}dx.$$
Observe that $C_r$ is analytic for $\Re(r)>1/2$ and has an analytic
continuation to a neighbourhood of $1/2$ where it has a simple pole (cf
\cite{CPRS2}) with residue equal to $1$.

\begin{lemma}  Let $\D$ be a
  self-adjoint operator on the Hilbert space $\HH$,  affiliated
  to the semifinite von Neumann algebra $\cM$. Suppose that for a
  fixed faithful, normal, semifinite trace $\tau$ on $\cM$ we have
$$(1+\D^2)^{-1/2}\in\LL^{(p,\infty)}(\cM,\tau),\quad p\geq 1.$$
Then the Laplace transform of $\eta_\epsilon(\D)$, the eta invariant of
$\D$, is given by $\frac{1}{C_r}\eta_\D(r)$ where
$$\eta_\D(r)=\int_1^\infty
\tau(\D(1+s\D^2)^{-r})s^{-1/2}ds,\quad \Re(r)>1/2+p/2.$$
\end{lemma}


\begin{proof} We need to Laplace transform the `$\theta$ summable
formula' for the truncated $\eta$ invariant:
$$\eta_\epsilon(\D)=\frac{1}{\sqrt{\pi}}
\int_\epsilon^\infty\tau(\D e^{-t\D^2})t^{-1/2}dt.$$
This integral converges for all $\epsilon>0$. First we rewrite the
formula as
$$\eta_\epsilon(\D)=\frac{\sqrt{\epsilon}}{\sqrt{\pi}}
\int_1^\infty\tau(\D e^{-\epsilon s\D^2})s^{-1/2}ds.$$
Using
$$1=\frac{1}{\Gamma(r-1/2)}
\int_0^\infty\epsilon^{r-3/2}e^{-\epsilon}d\epsilon$$
for $\Re(r)>p/2+1/2$, the Laplace transform of $\eta_\epsilon(\D)$ is
\begin{align}
\frac{1}{C_r}\eta_\D(r)&=\frac{1}{\sqrt{\pi}\Gamma(r-1/2)}
\int_0^\infty\epsilon^{r-1}e^{-\epsilon}\int_1^\infty
\tau(\D e^{-\epsilon s\D^2})s^{-1/2}dsd\epsilon\nno
&=\frac{1}{\sqrt{\pi}\Gamma(r-1/2)}\int_1^\infty s^{-1/2}
\tau(\D\int_0^\infty\epsilon^{r-1}e^{-\epsilon(1+s\D^2)}d\epsilon)ds\nno
&=\frac{\Gamma(r)}{\sqrt{\pi}\Gamma(r-1/2)}
\int_1^\infty s^{-1/2}\tau(\D(1+s\D^2)^{-r})ds.\end{align}
\end{proof}

For our final formula we restrict to $p=1$, which is the case of
interest in this paper.

\begin{prop}\label{residuespecflow}  Let $\D_a$ be a
self-adjoint densely defined unbounded operator on the Hilbert
space~$\HH$, affiliated to the semifinite von Neumann algebra $\cM$.
Suppose that for a  fixed faithful, normal, semifinite trace $\tau$ on
$\cM$ we have $(1+\D_a^2)^{-1/2}\in\LL^{(1,\infty)}(\cM,\tau).$ Let
$\D_b$ differ from $\D_a$ by a bounded self adjoint operator in  $\cM$.
Then for any piecewise $C^1$ path $\{\D_t=\D_a+A_{t}\},$ $t\in [a,b]$
in $\cM_0=\D_a+\cM_{sa}$ joining $\D_a$ and $\D_b$, the spectral flow $sf_\tau(\D_a,\D_b)$ is given by the
formula
\begin{align}
 \Res_{r=1/2}C_rsf_\tau(\D_a,\D_b)
&=\Res_{r=1/2}\left(\int_a^b\tau(\dot\D_t(1+\D_t^2)^{-r})dt+
\frac{1}{2}\left(\eta_{\D_b}(r)-\eta_{\D_a}(r)\right)\right)\nno
&+\frac{1}{2}\left(\tau(P_{\ker\D_b})-\tau(P_{\ker\D_a})\right),
\label{eq:resformula}
\end{align}
where $\eta_\D(r):= \int_1^\infty \tau(\D(1+s\D^2)^{-r})s^{-1/2}ds,\ \
\ \Re(r)>1.$ The meaning of \eqref{eq:resformula} is that
the function of~$r$ on the right hand side 
has a meromorphic continuation to a neighbourhood of $r=1/2$ with a
simple pole at $r=1/2$ where we take the residue.
\end{prop}

\begin{proof} We apply the Laplace transform to the general spectral
flow formula. The computation of the Laplace transform of the eta
invariants is above, and the Laplace transform of the other integral is
in \cite{CP2}. The existence of the residue follows from the equality,
for $\Re(r)$ large,
$$C_rsf_\tau(\D_a,\D_b)=\int_a^b\tau(\dot\D_t(1+\D_t^2)^{-r})dt+
\frac{1}{2}\left(\eta_{\D_b}(r)-\eta_{\D_a}(r)\right)+
C_r\frac{1}{2}\left(\tau(P_{\ker\D_b})-\psi(P_{\ker\D_a})\right)$$
which shows that the sum of the integral and the eta terms has a
meromorphic continuation as claimed.
\end{proof}

This is the formula for spectral flow we will employ in the sequel.

\subsection{An index formula for modular partial isometries}
Having obtained a well-defined analytic index pairing, we now emulate
the methods of \cite{CPRS2} to obtain a `local index formula' to
compute this analytic pairing.
Let $v\in A^\sim$ be a modular partial isometry. Recall that (as we
observed after Lemma~\ref{lemma:msf}) we automatically have $v_k\in\A$ for
$k\ne0$. Furthermore, the same Lemma shows that $v_0$ does not
contribute to the spectral flow. In other words, we have the following.

\begin{lemma}\label{lm:local-unit}  Given the
modular spectral triple for $(A, \sigma,\phi)$, let $v\in A^\sim$ be a
modular partial isometry so that $p=vv^*-v_0v_0^*\in\F$, where $v_0\in
A_0$ is the $\s$-invariant part of $v$. Then $p$ commutes with $\D$ and
$v\D v^*$, and
$$sf_{\phi_\D}(vv^*\D,v\D v^*)=sf_{\phi_{\D,p}}(p\D,pv\D v^*),$$
where $\phi_{\D,p}=\phi_\D|_{p\cM p}$ is the trace on $p\cM p$.
\end{lemma}

In our previous examples of the Cuntz algebra and $SU_q(2)$ we showed
that the operator $(1+\D^2)^{-1/2}$ lies in
$\LL^{(1,\infty)}(\cM,\phi_\D).$ However in general, by
Lemma~\ref{lemma:not-too-big2}, we only have
$f(1+\D^2)^{-1/2}\in\LL^{(1,\infty)}(\cM,\phi_\D)$ for $f\in\F=
F\cap\mbox{dom}(\phi)$. Thanks to the Lemma above this is sufficient
for our purposes, since  $p(1+\D^2)^{-1/2}$ is in
$\LL^{(1,\infty)}(p\cM p,\phi_{\D,p})$, and we are justified in using
our spectral flow formula.

We apply Proposition~\ref{residuespecflow} to the path
$\D_t=p\D+tpv[\D,v^*]=p\D+tv[\D,v^*]$ of operators affiliated with
$p\cM p$.

\begin{lemma}\label{lem:no-ker}
We have $\phi_{\D,p}(P_{\ker \D_0})-\phi_{\D,p}(P_{\ker\D_1})=0$.
\end{lemma}

\begin{proof}
Since we cut down by the projection $p$, we may assume that $v_0=0$ and
so $v\in\A$ and $vv^*=p$. Then $P_{\ker\D_0}=vv^*\Phi_0$ and
$P_{\ker\D_1}=v\Phi_0v^*$. As $\phi_\D(f\Phi_0)=\phi(f)$ for $f\in \F$
by Equation~\eqref{tr}, we have
$$
\phi_\D(vv^*\Phi_0-v\Phi_0 v^*)
=\phi_\D((\s_{-i\beta}(v^*)v-vv^*)\Phi_0)
=\phi(\s_{-i\beta}(v^*)v-vv^*)=0.
$$
\end{proof}

Thus the kernel correction terms vanish for modular partial isometries.
Next we obtain a residue formula for the spectral flow:

\begin{thm}\label{thm:analytic-index}
Given the modular spectral triple for $(A,\s,\phi)$ let $v\in \A^\sim$
be a modular partial isometry. Then
$sf_{\phi_{\D}}(vv^*\D,v\D v^*)$ is given by
$$\Res_{r=1/2}\left(r\mapsto \phi_\D(v[\D,v^*](1+\D^2)^{-r})+
\frac{1}{2}\int_1^\infty \phi_{\D}((\s_{-i\beta}(v^*)v-vv^*)\D
(1+s\D^2)^{-r})s^{-1/2}ds\right).$$
\end{thm}

\begin{proof}
We apply Proposition~\ref{residuespecflow} to the path
$\D_t=p\D+tv[\D,v^*]$. Thus by Lemma~\ref{lm:local-unit} and
Lemma~\ref{lem:no-ker} we get
$$
sf_{\phi_{\D}}(vv^*\D,v\D v^*)=\Res_{r=1/2}
\left(\int_0^1\phi_\D(v[\D,v^*](1+\D_t^2)^{-r})dt+
\frac{1}{2}\left(\eta_{\D_1}(r)-\eta_{\D_0}(r)\right)\right).
$$

First we observe that by \cite[Proposition 10, Appendix B]{CP1}, the
difference
$$(1+(\D+tv[\D,v^*])^2)^{-r}-(1+\D^2)^{-r}$$
is (uniformly) trace class in the corner $p{\mathcal M}p$ for $r\geq
1/2$. Hence in the spectral flow formula above we may exploit
analyticity in $r$ for $\Re(r)> 1/2$ as in \cite{CPRS2} (we are working
in the semifinite algebra $p\cM p$ with trace $\phi_\D|_{p\cM p}$) to
write
$$ \int_0^1\phi_\D(v[\D,v^*](1+(\D+tv[\D,v^*])^2)^{-r})dt
=\phi_\D(v[\D,v^*](1+\D^2)^{-r}) +\mbox{remainder}.$$ The remainder is
finite at $r=1/2$, and in fact by \cite{CPRS2}, holomorphic at $r=1/2$.

Next consider the eta terms. We have, for $\Re(r)>1$,
\begin{align*}
\eta_{\D_1}(r)&=
\int_1^\infty \phi_{\D}(pv\D v^*(1+s(v\D v^*)^2)^{-r})s^{-1/2}ds\nno
&=\int_1^\infty \phi_{\D}(pv\D (1+s\D^2)^{-r}v^*)s^{-1/2}ds\nno
&=\int_1^\infty \phi_{\D}(\s_{-i\beta}(v^*)pv\D (1+s\D^2)^{-r})s^{-1/2}ds
\end{align*}
and
$$
\eta_{\D_0}(r)=
\int_1^\infty \phi_{\D}(p\D (1+s\D^2)^{-r})s^{-1/2}ds.
$$
Using that $\s_{-i\beta}(v^*)pv=\s_{-i\beta}(v^*)v-v_0^*v_0$ and
$p=vv^*-v_0v_0^*$, we see that to finish the proof we have to check
that
$$
\phi_{\D}((v_0^*v_0-v_0v_0^*)\D (1+s\D^2)^{-r})=0.
$$
This is true since $\phi_{\D}(\cdot\,\D (1+s\D^2)^{-r})$ is a trace on
$\cM$  (note that if we considered partial isometries in a matrix
algebra over $\A^\sim$ we would have to require in addition that
$v_0^*v_0-v_0v_0^*$ is an element over~$\F$).
\end{proof}

Finally, when the circle action has full spectral subspaces, the eta
corrections also vanish.

\begin{cor}\label{cor:eta-ker}
Assume the circle action $\s$ has full spectral subspaces. Then for
every modular partial isometry $v\in A^\sim$ we have
$$ sf_{\phi_{\D}}(vv^*\D,v\D v^*)
=\Res_{r=1/2}\phi_\D(v[\D,v^*](1+\D^2)^{-r})dt.$$
\end{cor}

\begin{proof}
Consider the modular partial isometry $w=v-v_0$. Then $w\in\A$, so we
can apply the previous Theorem. Since the spectral flow corresponding
to $v$ and $w$ coincide and $v[\D,v^*]=w[\D,w^*]$, all we have to do is
to show that the eta term defined by $w$ vanishes. By the assumption of
full spectral subspaces we have
$$ \phi_\D((\s_{-i\beta}(w^*)w-ww^*)\Phi_k)=
\phi(\sigma_{-i\beta}(w^*)w-ww^*)=0,$$ for all $k\in\Z$, and as
$$
\phi_{\D}((\s_{-i\beta}(w^*)w-ww^*)\D (1+s\D^2)^{-r})=
\sum_{k\in\Z}\phi_\D((\s_{-i\beta}(w^*)w-ww^*)\Phi_k) k (1+sk^2)^{-r},
$$
the eta term is indeed zero.
\end{proof}

\subsection{Twisted cyclic cocycles}

This subsection is motivated by the observation of \cite{CPR2} that
when there are no eta or kernel correction terms we can define a
functional on $\A\otimes\A$ by
$$(a_0,a_1)\mapsto\omega\mbox{-}\!\lim_{s\to\infty}
\frac{1}{s}\phi_\D(a_0[\D,a_1](1+\D^2)^{-1/s-1/2})$$ which is, at
least formally, a twisted (by $\sigma_{-i\beta}$) cyclic cocycle.
However we saw in \cite{CRT} that in the case of $SU_q(2)$ the eta
corrections created a subtle difficulty in that individually they do
not have the same holomorphy properties as the term in the previous
equation and that only by combining them do we obtain something we can
understand in cohomological terms. Thus we set, for $a_0,a_1\in\A$,
$$
\eta_\D^r(a_0,a_1)=\frac{1}{2}\int_1^\infty
\phi_\D((\s_{-i\beta}(a_1)a_0-a_0a_1)
\D(1+s\D^2)^{-r})s^{-1/2}ds.
$$
This is well-defined for $\Re(r)>1$, and as we shall see later,
extends analytically to $\Re(r)>1/2$. When we pair with a modular
partial isometry we  necessarily have $(r-1/2)\eta^r_\D(v,v^*)$
bounded, since the sum of the eta term and
$\phi_\D(v[\D,v^*](1+\D^2)^{-r})$ has a simple pole by Proposition
\ref{residuespecflow}.

Throughout this Section, $b^\s,\,B^\s$ denote the twisted Hochschild
and Connes coboundary operators in twisted cyclic theory, \cite{KMT}.
The twisting will always come from the regular automorphism
$\s:=\s_{-i\beta}=\s_{i}^{\phi_\D}$ of $\A$ (recall that an algebra
automorphism $\s$ is regular if $\s(a)^*=\s^{-1}(a^*)$, \cite{KMT}).

In order to be able to describe the index pairing of Theorem
\ref{thm:analytic-index} as the pairing of a twisted $b^\s,B^\s$
cocycle with the modular $K_1$ group, we need to address the analytic
difficulties we have just described. This is done in the next Lemma.

\begin{lemma} \label{lemma:anamess}
For $a_0,\,a_1\in\A$, let
$$\psi^r(a_0,a_1)=\phi_\D(a_0[\D,a_1](1+\D^2)^{-r})+\eta^r_\D(a_0,a_1).$$
Then for  $a_0,a_1,a_2\in \A$ the functions $r\mapsto
\phi_\D(a_0[\D,a_1](1+\D^2)^{-r})$ and $r\mapsto \eta^r_\D(a_0,a_1)$
are analytic for $\Re(r)>1/2$, while
$r\mapsto(b^\s\psi^r)(a_0,a_1,a_2)$ is analytic for $\Re(r)>0$.
\end{lemma}

\begin{proof}
Recall that the algebra $\A$ consists of finite sums of homogeneous
elements in the domain of~$\phi$. Therefore we may assume that
$a_0,a_1,a_2$ are homogeneous. Consider the conditional expectation
$\Psi\colon\cn\to\cn^\s$, $\Psi(x)=\sum_n\Phi_nx\Phi_n$. Then
$\phi_\D=\phi_\D\circ\Psi$. It follows that if $a_0\in A_k$ and $a_1\in
A_m$ then $\phi_\D(a_0[\D,a_1](1+\D^2)^{-r})=0$ unless $k=-m$, and in
the latter case we have
$$
\phi_\D(a_0[\D,a_1](1+\D^2)^{-r})=\sum_{n\in\Z}\frac{s_n}{(1+n^2)^r},
$$
where $s_n=m\,\phi_\D(a_0a_1\Phi_n)$. By Lemma~\ref{lemma:not-too-big2}
the sequence $\{s_n\}_n$ is bounded. Hence the function
$\phi_\D(a_0[\D,a_1](1+\D^2)^{-r})$ is analytic for $\Re(r)>1/2$.

Consider now $\eta^r(a_0,a_1)$. If $a_0\in A_k$ and $a_1\in A_m$ then
$\eta^r(a_0,a_1)=0$ unless $k=-m$. In the latter case put
$s_n=\phi_\D(a_0a_1\Phi_n)$. Notice that
$$
\phi_\D(\s(a_1)a_0\Phi_n)=\phi_\D(a_0\Phi_na_1)
=\phi_\D(a_0a_1\Phi_{n-m})=s_{n-m}.
$$
The sequence $\{s_n\}_n$ is bounded.
Assume $m\ge0$. Then for
$\Re(r)>1$ we have
$$
\int_1^\infty \phi_\D((\s(a_1)a_0-a_0a_1)
\D(1+s\D^2)^{-r})s^{-1/2}ds\\
=\sum_{n\in\Z} \int_1^\infty \frac{(s_{n-m}-s_n)n}{(1+sn^2)^{r}}s^{-1/2}ds$$
which we may write as
\begin{align*}
&2\sum_{n>0}(s_{n-m}-s_n)\int^\infty_n\frac{dt}{(1+t^2)^r}
-2\sum_{n<0}(s_{n-m}-s_n)\int^\infty_{-n}\frac{dt}{(1+t^2)^r}\\
&=2\sum_{n=-m+1}^0s_n\int^\infty_{n+m}\frac{dt}{(1+t^2)^r}
-2\sum_{n>0}s_n\int^{n+m}_n\frac{dt}{(1+t^2)^r}\\
&+2\sum_{n=-m}^{-1}s_n\int^\infty_{-n}\frac{dt}{(1+t^2)^r}
-2\sum_{n<-m}s_n\int^{-n}_{-n-m}\frac{dt}{(1+t^2)^r}.\\
\end{align*}
The above series of functions analytic on $\Re(r)>1/2$ converge
uniformly on $\Re(r)>1/2+\epsilon$ for every $\epsilon>0$. A similar
argument works for $m\le0$. Hence the function
$r\mapsto\eta^r(a_0,a_1)$ extends analytically to $\Re(r)>1/2$.

Turning to $b^\s\psi^r$, first notice that $b^\s\eta^r_\D=0$, since
$r\mapsto b^\s\eta^r_\D$ is analytic for $\Re(r)>1/2$ and
$\eta^r_\D=b^\s\theta^r_\D$ for $\Re(r)>1$, where
$$
\theta^r_\D(a_0)=-\frac{1}{2}\int_1^\infty
\phi_{\D}(a_0\D(1+s\D^2)^{-r})s^{-1/2}ds.
$$
It follows that $(b^\s\psi^r)(a_0,a_1,a_2)$ is given by
\begin{align*}
&\phi_\D(a_0a_1[\D,a_2](1+\D^2)^{-r}-
\phi_\D(a_0[\D,a_1a_2](1+\D^2)^{-r}) +
\phi_\D(\s(a_2)a_0[\D,a_1](1+\D^2)^{-r})\nno
&=-\phi_\D(a_0[\D,a_1]a_2(1+\D^2)^{-r})+
\phi_\D(\s(a_2)a_0[\D,a_1](1+\D^2)^{-r}).\nno
\end{align*}
If $a_0\in A_k$, $a_1\in A_l$ and $a_2\in A_m$, then the above
expression is zero unless $k+l+m=0$. In the latter case put
$s_n=l\,\phi_\D(a_0a_1a_2\Phi_n)$. Then a computation similar to that
for $\eta^r$ yields, for $\Re(r)>1/2$,
\begin{align*}
(b^\s\psi^r)(a_0,a_1,a_2)
&=\sum_{n\in\Z}s_n((1+(n+m)^2)^{-r}-(1+n^2)^{-r})\\
&=\sum_{n\in\Z}s_n(1+n^2)^{-r}\left(\left(1+\frac{2mn+m^2}{1+n^2}\right)^{-r}-1\right)
\end{align*}
Using that if $\Omega$ is a compact subset of $\Re(r)>0$ then
$|(1+x)^{-r}-1|\le C\,|x|$ for some $C>0$, sufficiently small $x$ and
all $r\in\Omega$, we see that the above series converges uniformly on
$\Omega$. Hence $(b^\s\psi^r)(a_0,a_1,a_2)$ extends analytically to
$\Re(r)>0$.
\end{proof}

The following result links our analytic constructions to twisted cyclic
theory.
\begin{prop}\label{pr:almost-cocycle}
Given the modular spectral triple for $(A, \sigma,\phi)$ define
a bilinear functional on $\A$ with values in the functions holomorphic for
$\Re(r)>1$ by
$$ a_0,a_1\mapsto
\left(r\mapsto\left(\phi_\D(a_0[\D,a_1](1+\D^2)^{-r})+
\frac{1}{2}\int_1^\infty \phi_{\D}((\s(a_1)a_0-a_0a_1)
\D(1+s\D^2)^{-r})s^{-1/2}ds\right)\right)
 $$
This functional continues analytically to $\Re(r)>1/2 $ and is a twisted $b,B$-cocycle  modulo functions holomorphic for
$\Re(r)>0$.
The twisting is given by the regular
automorphism~$\s:=\s_{-i\beta}=\s_{i}^{\phi_\D}$.
\end{prop}
{\bf Remark}. If $\phi$ is a state then the domain of the cocycle of
the Proposition is much larger, but to prove this requires more work.
\begin{proof}
As before, for $\Re(r)>1$ we define the functional $\psi^r$ by the
formula
$$\psi^r(a_0,a_1)
=\phi_\D(a_0[\D,a_1](1+\D^2)^{-r})+ \frac{1}{2}\int_1^\infty
\phi_{\D}((\s(a_1)a_0-a_0a_1)\D(1+s\D^2)^{-r})s^{-1/2}ds,$$ and then
extend $\psi^r$ analytically to $\Re(r)>1/2$, which is possible by
Lemma~\ref{lemma:anamess}. Then $(B^\s\psi^r)(a_0)=\psi^r(1,a_0)$ and
for $\Re(r)>1$ is given by
$$(B^\s\psi^r)(a_0)=\phi_\D([\D,a_0](1+\D^2)^{-r})+
\frac{1}{2}\int_1^\infty \phi_{\D}((\s(a_0)-a_0)\D(1+s\D^2)^{-r})s^{-1/2}ds.$$
The first term vanishes since $\Psi([\D,a_0])=0$ for any $a_0\in\A$,
while the second terms vanishes by $\s_t$-invariance of $\phi_\D$. That
$b^\s\psi^r$ is analytic for $\Re(r)>0$ was proved in the last Lemma.
\end{proof}
\begin{cor} \label{cor:rescocycle}
If the circle action has full spectral subspaces then for all
$a_0,a_1\in\A$ the residue
$$\phi_1(a_0,a_1):=\Res_{r=1/2}\phi_\D(a_0[\D,a_1](1+\D^2)^{-r})
$$
exists and equals $\phi(a_0[\D,a_1])$. It defines a twisted cyclic
cocycle on $\A$, and for any modular partial isometry $v\in\A$
$$sf_{\phi_\D}(vv^*\D,v\D v^*)=\phi_1(v,v^*)=\phi(v[\D,v^*]).$$
\end{cor}
\begin{proof}
Under the full spectral subspaces assumption we have
$\phi_\D(f\Phi_n)=\phi(f)$ for $f\in\F$, whence
$$
\phi_\D(a_0[\D,a_1](1+\D^2)^{-r})=\phi_\D(\Psi(a_0[\D,a_1])(1+\D^2)^{-r})
=\phi(a_0[\D,a_1])\sum_{n\in\Z}\frac{1}{(1+n^2)^r}.
$$
This shows that the residue exists and equals $\phi(a_0[\D,a_1])$. That
it defines a twisted cyclic cocycle follows from the proof of
Proposition~\ref{pr:almost-cocycle}. That $\phi_1(v,v^*)$ computes the
spectral flow follows from Corollary~\ref{cor:eta-ker}.
\end{proof}
{\bf Remark}.  It is of course easy to see directly that
$\phi(a_0[\D,a_1])$ is a twisted cyclic cocycle, while the fact  that it computes
the spectral flow agrees with Lemma~\ref{lemma:msf}.

Finally, we have the following reformulation of Lemma
\ref{lemma:not-too-big2} as an analogue of the result of A. Connes on
continuity (in $\omega$) of the Dixmier trace on pseudodifferential
operators on a compact manifold in the KMS-weight context.

\begin{prop}\label{cr:dixy-comp}
Given the modular spectral triple for $(A, \sigma,\phi)$ and a Dixmier
functional (see Section 3 of \cite{CPS2}), $\omega\in
(L^\infty(\R))^*$, define $\phi_{\D,\omega}: A_+\to[0,\infty]$ by
$$\phi_{\D,\omega}(a)
=\omega\mbox{-}\!\lim_{r\to
\infty}\frac{1}{r}\phi_\D(\Psi(a)(1+\D^2)^{-1/2-1/2r}),
$$
where $\Psi\colon\cn\to\cM$ is the $\phi_\D$-preserving conditional
expectation. Then $\phi_{\D,\omega}(a)\le2\phi(a)$ for any $a\in
A_+\cap{\rm dom}(\phi)$, with equality  if $A$ has full spectral
subspaces. In particular, when the circle action has full spectral
subspaces, $\phi_{\D,\omega}(a)$ is independent of $\omega$.
\end{prop}

\begin{proof}
Recall that $\Psi(x)=\sum_{n\in\Z}\Phi_nx\Phi_n$. Observe also that
$\Phi_na\Phi_n=\Phi(a)\Phi_n$. Hence for $a\in A_+$,
$$\phi_\D(\Phi_na\Phi_n)=\phi_\D(\Phi(a)\Phi_n)\leq\phi(\Phi(a))=\phi(a)$$
by Lemma~\ref{lemma:not-too-big2}, with equality if $A$ has full
spectral subspaces. This shows that
$$\phi_{\D,\omega}(a)=\omega\mbox{-}\!\lim_{r\to
\infty}\frac{1}{r}\sum_{n\in\Z}\phi_\D(\Phi_na\Phi_n)
(1+n^2)^{-1/2-1/2r}$$
exists  and when $A$ has full spectral subspaces
$$\phi_{\D,\omega}(a)=\lim_{r\to
\infty}\frac{\phi(a)}{r}\sum_{n\in\Z}
(1+n^2)^{-1/2-1/2r}=2\phi(a).$$
\end{proof}

\section{Examples}

Our first two examples are covered in detail in \cite{CPR2,CRT} so we will
only present a summary here.


{\bf Example 1}. For the algebra $\mathcal O_n$ (with generators
$S_1,\ldots ,S_n$) we write $S_\alpha$  for the product
$S_{\mu_1}\ldots S_{\mu_k}$ and $k=|\alpha|$. We take the usual gauge
action $\s$, and the unique KMS state $\phi$ for this circle action.

In the Cuntz algebra case we have full spectral subspaces. Due to the
absence of eta terms, the analytic formula is the easiest to apply, so
we can compute the pairing with $S_\alpha S_\beta^*$ using the residue
cocycle, Corollary~\ref{cor:rescocycle}, and get
$$
 sf(S_\alpha S_\alpha^*\D, S_\alpha S_\beta^*\D S_\beta S_\alpha^*)
=(|\beta|-|\alpha|)\frac{1}{n^{|\alpha|}}.
$$

{\bf Example 2}.
For $SU_q(2)$ we used the graph algebra description of Hong and Szymanski,
\cite{HS}, and we use the notation and  computations
from \cite{CRT}. There we introduced a new
set of generators
$T_k,\,\tilde{T}_k,\,U_n$ for this algebra.
The generators  $T_k$ and $\tilde{T}_k$ are
non-trivial homogenous partial isometries for the modular
group of the Haar state, $h$, which is a KMS$_{-\log q^2}$ state.

For $SU_q(2)$ there are eta correction  terms. Given the explicit
computations in \cite{CRT} and the description of the fixed point
algebra as the unitization of an infinite direct sum of copies of
$C(S^1)$ (that is the $C^*$-algebra of the one point compactification
of an infinite union of circles of radius $q^{2k}$, $k\geq 0$) it is
not hard to see that our SSA is satisfied for $SU_q(2)$.

The presence of the eta corrections makes the analytic computation of
spectral flow from the twisted cocycle harder (it can still be done explicitly as in
\cite{CRT}). Instead we employ the factorisation through the $KK$-pairing. Taking the value
of the trace $\mbox{Tr}_\phi(T_k^*T_k\Phi_j)=q^{2(|j|+1)}$ from \cite{CRT} we have
\begin{align*} sf_{\phi_\D}(T_k^*T_k\D,T^*_k\D T_k)&=\Ev(e^{\log q^2})\circ\tau_*
\left(\sum^{-1}_{j=-k} [T_k^*T_k\Phi_j]\chi^j\right)
=\sum^{-1}_{j=-k} \mbox{Tr}_\phi(T_k^*T_k\Phi_j)q^{2j}
=kq^2.
\end{align*}

The point of this example is  that there are naturally occurring
examples satisfying the SSA but without full spectral subspaces.

{\bf Example 3: the Araki-Woods factors}. We will follow the treatment
of the Araki-Woods factors in Pedersen~\cite{Ped}, Subsection 8.12 and
the subsequent discussion. We let $A$ be the Fermion algebra, that is
the $C^*$-inductive limit of the matrix algebras $\Mat_{2^n}(\C)$ which
is the $n$-fold tensor product of the matrix algebra of $2\times 2$
matrices $\Mat_2(\C)$.

For $0<\lambda<1/2$ let
$$h_n=\otimes_{k=1}^n\ \left(\begin{array}{cc} 2(1-\lambda) &  0  \\
                             0 &  2\lambda\end{array} \right)
$$
Let $\phi$ be the tracial state on $A$ (given by the tensor product of
the normalised traces on $\Mat_2(\C)$) and define
$$\phi_\lambda(x)=\phi(h_nx),\ x\in \Mat_{2^m}(\C),\ m\leq n.$$
Then $\phi_\lambda$ is a state on $\Mat_{2^m}(\C)$ and is independent
of $n$. By continuity it extends to a state on~$A$. Consider the
automorphism group defined by $\Ad h_n^{-it}$. It is not hard to see
that $\phi_\lambda$ satisfies the KMS condition with respect to $\Ad
h_n^{-it}$ at $1$ for this group or equivalently at
$\beta=\ln\frac{1-\lambda}{\lambda}$ for the gauge action $\s_t=\Ad
h_n^{-it/\beta}$.  Everything extends by continuity to $A$. Then
the GNS representation
corresponding to  $\phi_\lambda$ generates a type III$_{\lambda'}$
factor, where $\lambda'=\lambda/(1-\lambda)$ (for a proof see \cite{Ped} 8.15.13).

The simplest way to see the we have full spectral subspaces for the
circle action $\s$ is to replace this version of the Fermion algebra by
the isomorphic copy given by annihilation and creation operators, see
e.g.~\cite{EK}.

To describe the isomorphism, we let $\sigma_j,\ j=1,2,3$ be the Pauli
matrices in their usual representation:
$$ \s_1=\bma 0 & 1\\1&0\ema,\qquad \s_2=\bma 0 & i\\ -i & 0\ema,\qquad
\s_3=\bma -1 & 0\\ 0 & 1\ema.$$ Then the isomorphism is given by
defining $a_j=\sigma_3\otimes\ldots \sigma_3\otimes
(\sigma_1+i\sigma_2)/2$, where the last term is in the $j$-th tensorial
factor. Then the $a_j$, $j\in\N$, and their adjoints $a_j^*$ satisfy
the usual relations of the $C^*$-algebra of the canonical
anticommutation relations (i.e. the Fermion algebra):
$$a_ja_k^*+a_k^*a_j=\delta_{jk}, \ \ \ a_ja_k+a_ka_j=0.$$
The gauge invariant algebra is generated by monomials in the
$a_j,a_k^*$ which have equal numbers of creation and annihilation
operators. Clearly ${A}_1$ is generated by monomials with one more
creation operator than annihilation operator. From the anticommutation
relations above it is now clear that ${A}_1^*{A}_1$ and $A_1A_1^*$ are
dense in the gauge invariant subalgebra. So we have full spectral
subspaces. Thus the main results of the paper apply to this example.

Modular partial isometries are easy to find, since each $a_j$ is an
homogenous partial isometry in $A_1$.  For a single $a_j$ we can employ
the twisted cyclic cocycle to get the index
\begin{align*}
sf_{\phi_\D}(a_ja_j^*\D,a_j\D
a_j^*)=-\phi(a_ja_j^*)
=-\lambda
=-(1+e^\beta)^{-1}.
\end{align*}
Similarly if we have the partial isometry $v$ formed by taking the
product of $n$ distinct $a_j$'s we obtain
$$sf_{\phi_\D}(vv^*\D,v\D v^*)=-n(1+e^\beta)^{-n}.$$
In \cite{CPR2} we made the observation that for modular unitaries
$u_v$, $sf_{\phi_\D}(\D,u_{v}\D u_{v}^*)$ is just Araki's relative
entropy \cite{Ar} of the two KMS weights $\phi_\D$ and $\phi_\D\circ
\Ad u_v$. In this example of the Fermion algebra we see that the
relative entropy depends on two physical parameters, the inverse
temperature $\beta$ and the modulus of the charge $n$ carried by the
product of Fermion annihilation or creation operators appearing in $v$.

\end{document}